\documentclass[11pt,a4paper,reqno]{amsart}
\usepackage{amsmath,amsthm,amsfonts,amssymb,bm,wasysym}
\usepackage{comment}
\usepackage{caption}
\usepackage{subcaption}
\usepackage{fullpage,bbm}
\usepackage{graphicx}
\usepackage{tikz,ifthen}
\usepackage{hyperref}
\usetikzlibrary{calc}
\usetikzlibrary{patterns}
\usetikzlibrary{arrows}
\usetikzlibrary{decorations.pathreplacing}
\theoremstyle{plain}
\newtheorem{theorem}{Theorem}
\newtheorem{proposition}[theorem]{Proposition}

\newtheorem{lemma}[theorem]{Lemma}

\theoremstyle{definition}
\newtheorem{definition}[theorem]{Definition}

\theoremstyle{remark}

\begin{document}

\title{On maximal hard-core thinnings of stationary particle processes}
\author{Christian Hirsch}
\thanks{This research publication was funded by LMU Munich's Institutional Strategy LMUexcellent within the framework of the German Excellence Initiative.
    }
\address{Mathematisches Institut, Ludwig-Maximilians-Universit\"at M\"unchen, 80333 Munich, Germany}
\email{hirsch@math.lmu.de} 
\author{G\"unter Last}
\address{Institute of Stochastics, Karlsruhe Institute of Technology, 76128 Karlsruhe, Germany}
              \email{guenter.last@kit.de}

\def\MLine#1{\par\hspace*{-\leftmargin}\parbox{\textwidth}{\[#1\]}}

\def\d{{\rm d}}
\def\E{\mathbb{E}}
\def\es{\emptyset}
\def\intt{\mathsf{int}}
\def\mc{\mathcal}
\def\ms{\mathsf}
\def\LPhi{\mc{L}(\Phi)}
\def\one{\mathbbmss{1}}
\def\Q{\mathbb{Q}}
\def\P{\mathbb{P}}
\def\R{\mathbb{R}}
\def\Z{\mathbb{Z}}

\maketitle

\begin{abstract}
The present paper studies existence and distributional uniqueness of subclasses of stationary hard-core particle systems arising as thinnings of stationary particle processes. These subclasses are defined by natural maximality criteria. We investigate two specific criteria, one related to the intensity of the hard-core particle process, the other one being a local optimality criterion on the level of realizations. In fact, the criteria are equivalent under suitable moment conditions. We show that stationary hard-core thinnings satisfying such criteria exist and are frequently distributionally unique. More precisely, distributional uniqueness holds in subcritical and barely supercritical regimes of continuum percolation. Additionally, based on the analysis of a specific example, we argue that fluctuations in grain sizes can play an important role for establishing distributional uniqueness at high intensities. Finally, we provide a family of algorithmically constructible approximations whose volume fractions are arbitrarily close to the maximum.
\end{abstract}

\section{Introduction}
\label{defSec}
Motivated by applications in materials science, the problem of finding good models for hard-core particle systems has a long history. Gibbs processes based on a suitable hard-core potential offer the possibility of formalizing the heuristic of a Boolean model conditioned on a certain hard-core constraint. However, Gibbsian particle processes are notoriously difficult to simulate~\cite{mase}. Additionally, it seems debatable whether the approach of conditioning on the hard-core constraint is a reasonable approximation to the physical mechanisms that lead to the formation of hard-core systems in the microstructure of advanced materials. 
Another popular and natural possibility to create hard-core particle processes starts from a random initial configuration of particles that may exhibit overlappings. Then, the hard-core constraint is enforced via a suitable thinning. For instance, the classical Mat\'ern-type processes are obtained from a Boolean model by applying an appropriate thinning rule.

Although models of Mat\'ern-type are appealingly simple to define, they suffer from the drawback of achieving only moderately high intensities. In other words, many packings appearing in materials science exhibit a substantially higher volume fraction. We refer the reader to~\cite{torquato} for a more detailed discussion concerning the relevance of random close packings in materials science. Of course, Mat\'ern-type thinnings give only specific examples of thinning operations. Is it possible to achieve denser packings by using different kinds of thinning mechanisms?

This question has recently been addressed systematically by investigating the subclass of hard-core thinnings of stationary particle processes that maximize the intensity or volume fraction under a hard-core constraint~\cite{hoerigPhD,hoerig}. Whereas the focus of~\cite{hoerig} is on simulation techniques for bounded sampling windows, in the present paper we investigate the extension of such volume-maximizing thinnings to stationary particle processes defined on the entire Euclidean space. In particular, we substantially extend the results of~\cite{hoerigPhD}.

Let $\mc{K}$ denote the family non-empty compact subsets of $\R^d$ and observe that the Hausdorff distance endows $\mc{K}$ with the structure of a metric space~\cite{sWeil}. We let $\mc{B}(\mc{K})$ denote the Borel $\sigma$-algebra on $\mc{K}$. Furthermore, $\Phi$ is assumed to be a stationary particle process on some probability space $(\Omega_0,\mc{F}_0,\P_0)$, i.e., $\Phi$ is a stationary point process on $\mc{K}$. We assume the intensity $\gamma_0$ of $\Phi$ to be non-zero and finite. 

In the following, we consider stationary thinnings of the particle process $\Phi$. To make this more precise, we let $\mc{T}_{}=\mc{T}_\Phi$ denote the family of particle processes $\Psi$ defined on some probability space $(\Omega,\mc{F},\P)$ such that there exists a measurable map $S:\Omega\to\Omega_0$ with the following properties:
\begin{enumerate}
	\item $\P_0$ is the image measure of $\P$ under $S$,
	\item $\Psi$ is a realizationwise subset of $\Phi\circ S$, and
	\item the particle processes $\Psi$ and $\Phi\circ S$ are jointly stationary.
\end{enumerate}
A priori, the probability space $(\Omega, \mc{F}, \P)$ could vary from one particle processes in $\mc{T}_\Phi$ to another. However, after a possible extension, all random variables occurring in this paper can be assumed to be defined on a fixed probability space $(\Omega, \mc{F}, \P)$ with expectation operator $\E$. Therefore, to ease notation, we shall assume that $\Psi$ and $\Phi$ are defined on this space with $S$ being the identity map. Note that $(\Phi,\Psi)$ is a coupling of $\Psi$ and $\Phi$.


We write $\Q$ to denote the typical grain distribution of $\Psi$. Introducing the intensity $\gamma$ of $\Psi$ via
$$\gamma = \E\sum_{K\in\Psi}\one\{c(K) \in [0,1]^d\} ,$$
where $c(K)$ is the center of gravity of the particle $K$, this distribution is characterized by the Campbell formula
$$\E \sum_{K\in\Psi} f(K) = \gamma \int_{\R^d}\int_{\mc{K}}f(K + x)\Q(\d K) \d x$$
for any measurable function $f:\mc{K}\to[0,\infty)$.  Moreover, we assume $\gamma$ to be finite throughout the manuscript. Then, for a measurable and translation invariant function $h:\mc{K}\to[0,\infty)$ we let 
	$$\gamma_h(\Psi)=\gamma \int h(K)\Q(\d K)$$
denote the \emph{$h$-intensity} of $\Psi$. For instance, if $h\equiv1$, then $\gamma_h(\Psi)$ is just the intensity of $\Psi\in\mc{T}_{}$.

A particle configuration $\varphi$ \emph{satisfies the hard-core constraint}, in symbols $\varphi\in E_{\ms{hc}}$, if and only if the interiors of particles in $\varphi$ are pairwise non-overlapping. That is, $\intt(K)\cap\intt(K')=\es$ for all distinct $K,K'\in\varphi$. 
Then, $\mc{T}_{\ms{hc}}=\mc{T}_{\Phi,\ms{hc}}\subset\mc{T}_{}$ denotes the subset of $\mc{T}_{}$ consisting of the stationary thinnings $\Psi$ with $\P(\Psi\in E_{\ms{hc}})=1$. In other words, elements of $\mc{T}_{\ms{hc}}$ describe stationary hard-core thinnings. If $\Psi\in\mc{T}_{\ms{hc}}$, by choosing $h=\lambda_d$ to be the Lebesgue measure in $\R^d$, the volume fraction is another special case of the $h$-intensity.

In this paper, we investigate processes in $\mc{T}_{\ms{hc}}$ exhibiting certain maximality properties. More precisely, we consider \emph{intensity maximal} and \emph{locally maximal} thinnings.

Loosely speaking, intensity-maximal thinnings are elements $\Psi \in \mc{T}_{\ms{hc}}$ with maximum possible $h$-intensity.  Here, the \emph{maximum $h$-intensity}
$$\gamma_{h,\ms{max}}=\gamma_{\Phi,h,\ms{max}}=\sup_{\Psi\in\mc{T}_{\ms{hc}}}\gamma_h(\Psi)$$
is the supremum over all $h$-intensities of stationary hard-core thinnings of $\Phi$. The thinning $\Psi\in\mc{T}_{\ms{hc}}$ is \emph{$h$-intensity maximal} if $\gamma_h(\Psi)=\gamma_{h,\ms{max}}$. An illustration of a volume-maximal thinning based on a cut-out of a Poisson Boolean model consisting of disks attached to the points of a Poisson point process is shown in Figure~\ref{simFig}.

Note that we do not require maximal thinnings to be factors (in the sense of~\cite{tmatch}) of the underlying particle process. That is, we do not require that there exists a deterministic and translation-covariant algorithm extracting the thinning from the underlying particle process. 

In particular, Mat\'ern I hard-core processes are elements of $\mc{T}_{\ms{hc}}$, so that $\gamma_{h,\ms{max}}$ is positive.
 In the following, we let $\mc{T}_{h,\ms{i-max}}$ denote the subset of $\mc{T}_{\ms{hc}}$ consisting of all $h$-intensity maximal thinnings. 

\begin{figure}[!htpb]
 \centering 
\begin{tikzpicture}[scale=1.0]
\fill[black!30!white] (1.119,2.862) circle (0.571);
\draw (1.119,2.862) circle (0.571);
\fill[black!30!white] (0.028,3.191) circle (0.586);
\draw (0.028,3.191) circle (0.586);
\fill[black!30!white] (4.001,0.919) circle (0.590);
\draw (4.001,0.919) circle (0.590);
\fill[black!30!white] (4.496,3.976) circle (0.502);
\draw (4.496,3.976) circle (0.502);
\fill[black!30!white] (0.742,1.520) circle (0.522);
\draw (0.742,1.520) circle (0.522);
\fill[black!30!white] (4.225,2.044) circle (0.591);
\draw (4.225,2.044) circle (0.591);
\fill[black!30!white] (1.146,3.292) circle (0.579);
\draw (1.146,3.292) circle (0.579);
\fill[black!30!white] (1.591,0.752) circle (0.505);
\draw (1.591,0.752) circle (0.505);
\fill[black!30!white] (4.787,0.308) circle (0.515);
\draw (4.787,0.308) circle (0.515);
\fill[black!30!white] (4.663,1.027) circle (0.568);
\draw (4.663,1.027) circle (0.568);
\fill[black!30!white] (1.582,0.348) circle (0.554);
\draw (1.582,0.348) circle (0.554);
\fill[black!30!white] (3.801,4.268) circle (0.563);
\draw (3.801,4.268) circle (0.563);
\fill[black!30!white] (2.605,1.641) circle (0.531);
\draw (2.605,1.641) circle (0.531);
\fill[black!30!white] (2.113,4.973) circle (0.591);
\draw (2.113,4.973) circle (0.591);
\fill[black!30!white] (4.070,1.596) circle (0.598);
\draw (4.070,1.596) circle (0.598);
\fill[black!30!white] (2.033,4.562) circle (0.520);
\draw (2.033,4.562) circle (0.520);
\fill[black!30!white] (1.157,0.183) circle (0.537);
\draw (1.157,0.183) circle (0.537);
\fill[black!30!white] (1.618,0.288) circle (0.549);
\draw (1.618,0.288) circle (0.549);
\fill[black!30!white] (4.206,3.981) circle (0.525);
\draw (4.206,3.981) circle (0.525);
\fill[black!30!white] (0.948,2.512) circle (0.575);
\draw (0.948,2.512) circle (0.575);
\fill[black!30!white] (2.800,2.667) circle (0.591);
\draw (2.800,2.667) circle (0.591);
\fill[black!30!white] (4.507,3.798) circle (0.506);
\draw (4.507,3.798) circle (0.506);
\fill[black!30!white] (3.861,1.448) circle (0.529);
\draw (3.861,1.448) circle (0.529);
\fill[black!30!white] (0.546,3.813) circle (0.500);
\draw (0.546,3.813) circle (0.500);
\fill[black!30!white] (1.630,2.452) circle (0.526);
\draw (1.630,2.452) circle (0.526);
\fill[black!70!white] (3.407,3.753) circle (0.600);
\draw (3.407,3.753) circle (0.600);
\fill[black!70!white] (1.766,2.198) circle (0.557);
\draw (1.766,2.198) circle (0.557);
\fill[black!70!white] (2.404,1.018) circle (0.568);
\draw (2.404,1.018) circle (0.568);
\fill[black!70!white] (1.095,1.237) circle (0.581);
\draw (1.095,1.237) circle (0.581);
\fill[black!70!white] (3.157,4.944) circle (0.572);
\draw (3.157,4.944) circle (0.572);
\fill[black!70!white] (2.116,3.592) circle (0.598);
\draw (2.116,3.592) circle (0.598);
\fill[black!70!white] (4.592,3.437) circle (0.549);
\draw (4.592,3.437) circle (0.549);
\fill[black!70!white] (1.134,4.710) circle (0.573);
\draw (1.134,4.710) circle (0.573);
\fill[black!70!white] (4.359,4.648) circle (0.506);
\draw (4.359,4.648) circle (0.506);
\fill[black!70!white] (0.171,2.345) circle (0.580);
\draw (0.171,2.345) circle (0.580);
\fill[black!70!white] (0.779,3.434) circle (0.516);
\draw (0.779,3.434) circle (0.516);
\fill[black!70!white] (3.498,2.182) circle (0.573);
\draw (3.498,2.182) circle (0.573);
\fill[black!70!white] (4.476,0.071) circle (0.552);
\draw (4.476,0.071) circle (0.552);
\fill[black!70!white] (3.338,0.083) circle (0.581);
\draw (3.338,0.083) circle (0.581);
\fill[black!70!white] (4.705,1.942) circle (0.553);
\draw (4.705,1.942) circle (0.553);
\end{tikzpicture}
  \caption{Volume-maximal thinning on the cut-out of a Boolean model}
\label{simFig}  
\end{figure}
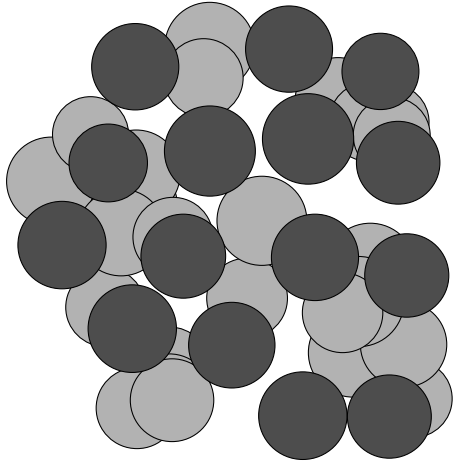

In our first main result, we use a subsequential limit argument to establish the existence of stationary $h$-intensity maximal hard-core thinnings. For a measurable and translation invariant function $h:\mc{K}\to[0,\infty)$, we let $\mc{K}_h \subset \mc{K}$ denote set of discontinuity points of $h$.
\begin{theorem}
\label{existenceProp}
    Assume that $\gamma_h(\LPhi)<\infty$ and that $\mc{K}_h$ is a zero-set with respect to $\Q$. Then, $\mc{T}_{h,\ms{i-max}}\ne\es$.
\end{theorem}

Ideally, we would like to have an explicit algorithm generating stationary maximal hard-core thinnings from a given configuration of particles. One idea could be to start from a suitable stationary tessellation and to optimize the configuration within the different cells separately. More precisely, assume that in addition to $\Phi$ there exists a random tessellation $\Xi$ such that $\Phi$ and $\Xi$ are defined on the same probability space and such that the pair $(\Phi,\Xi)$ is jointly stationary. Then, we consider the thinning $\Phi^-_{\Xi,\ms{max}}$ of $\Phi$ whose configuration in a cell $\Xi_i$ is described as follows. Starting from the family 
$$\Phi_{\Xi_i}^-=\{K\in\Phi:\, K\subset\Xi_i\}$$
of all particles of $\Phi$ that are entirely contained within the cell $\Xi_i$, we choose the hard-core thinning $\Phi^-_{\Xi_i,\ms{max}}$ of $\Phi_{\Xi_i}^-$ achieving the maximal aggregate $h$-value. If there are several possible configurations achieving this maximal value, we choose one according to some translation invariant rule. 

Since the thinned particles are contained in their corresponding cells, assembling them into a single configuration $\Phi^-_{\Xi,\ms{max}}=\cup_{i\ge1}\Phi^-_{\Xi_i,\ms{max}}$ preserves the hard-core property. Since particles intersecting cell boundaries do not enter the optimization, this algorithm does not lead to an $h$-maximal thinning. Nevertheless, in the next result, which can be seen as generalization of~\cite[Satz 5.2.5]{hoerigPhD},  the maximal intensity is approached arbitrarily closely. We let $A \oplus A' = \{a + a':\, a \in A, a' \in A'\}$ denote the Minkowski sum of $A,A' \subset \R^d$. Additionally, for 
a family of stationary random tessellations $\{\Xi(k)\}_{k\ge1}$, certain typical isoperimetric-type coefficients of the form 
\begin{align}
	\label{isoEq}
	\frac{\E\lambda_d(\partial\Xi(k)^*\oplus[-m,m]^d)}{\E\lambda_d(\Xi(k)^*)},
\end{align}
are considered, where $\Xi(k)^*$ denotes the typical cell of the tessellation $\Xi(k)$.

\begin{theorem}
\label{approxThm}
    Let $\{\Xi(k)\}_{k\ge1}$ be a family of stationary random tessellations such that for every $k\ge1$ the pair $(\Xi(k),\Phi)$ is jointly stationary. Also assume that for every $m\ge1$ the typical isoperimetric-type coefficients in~\eqref{isoEq} tend to 0 as $k \to \infty$. Then,
$$\lim_{k\to\infty}\gamma_h(\Phi^-_{\Xi(k),\ms{max}})=\gamma_{h,\ms{max}}.$$
\end{theorem}

In Section~\ref{approxProofSec}, we provide two examples of tessellations whose typical isoperimetric-type coefficients tend to 0. First, we consider Poisson-Voronoi tessellations of decreasing intensities whose underlying Poisson point process is assumed to be independent of the particle process $\Phi$. 

As a second example, we take up the construction from~\cite{timar} and consider a family of stationary Voronoi tessellations whose process of cell centers is a factor of the particle process $\Phi$, in the sense that it can be expressed as a measurable function of $\Phi$.

As an alternative to describing maximal hard-core thinnings via their intensity, we now propose a realizationwise characterization based on a local maximality property. Loosely speaking, swapping a finite number of grains in the thinning by a finite number of grains outside the thinning should not lead to a net increase in $h$-values if the swap preserves the hard-core property.
To be more precise, let $\psi,\varphi$ be locally finite configurations of convex grains such that $\psi \in E_{\ms{hc}}$ and $\psi\subset\varphi$. We say that $\psi$ is a \emph{locally $h$-maximal thinning of $\varphi$} (short: \emph{locally $h$-maximal}) if whenever $\psi'\in E_{\ms{hc}}$ is such that $\psi'\subset\varphi$, $\psi\Delta\psi'=(\psi\setminus \psi')\cup(\psi'\setminus \psi)$ is finite and $\psi'\ne\psi$, then
$$\sum_{K \in \psi'\setminus \psi}h(K)<\sum_{K \in \psi\setminus \psi'}h(K).$$
A stationary random thinning is \emph{locally maximal} if and only if the event of being locally maximal has probability $1$.

Next, we show that under a suitable moment condition, almost sure local maximality and intensity-maximality are equivalent. We note that quite similar moment conditions appear naturally in the investigation of densities of additive functionals for particle processes, see~\cite[Section 9.2]{sWeil}.  In the following, $\mc{T}_{h,\ms{\ell-max}}$ denotes the subset of $\mc{T}_{\ms{hc}}$ consisting of all almost-surely locally $h$-maximal thinnings with respect to $\Phi\circ S$.
\begin{theorem}
\label{localOptThm}
If $\gamma_h(\LPhi) < \infty$, then $\mc{T}_{h,\ms{i-max}}\subset\mc{T}_{h,\ms{\ell-max}}$. Moreover, the identity 
$$\mc{T}_{h,\ms{i-max}}=\mc{T}_{h,\ms{\ell-max}}$$
 holds under the additional moment condition
\begin{align}
\label{momCond1}
\int \lambda_d(K\oplus [-1,1]^d)h(K)\Q(\d K)<\infty.
\end{align}
\end{theorem}

Now, we know that stationary $h$-maximal hard-core thinnings exist. But are they also unique in a distributional sense?
In addition to existence, it is natural to consider distributional uniqueness. 
In general, we suspect that $\mc{T}_{h,\ms{\ell-max}}$ can consist of infinitely many distinct distributions of particle processes for large intensities and advertise as an open problem the development of a general explicit  non-asymptotic algorithmic description for at least \emph{some} of its members. 
Nevertheless, in many situations it turns out that stationary $h$-maximal hard-core thinnings are in fact distributionally unique. 
As opposed to existence, the issue of distributional uniqueness is more complex and the rest of the present paper is devoted to this topic.

First, distributional uniqueness should hold in the subcritical regime of continuum percolation, as we can choose a maximal thinning in each of the finite clusters. To make this idea rigorous, we impose additional assumptions on the distribution of particles. Indeed, imagine a configuration of two overlapping particles of equal volume that are disjoint from all other particles. Then, we are left with a choice as regards to which of the two particles should be part of a volume-maximal thinning. Therefore, we assume that all factorial moment measures of the marked point process $\{(c(K),h(K)):\,K\in\Phi\}$ are absolutely continuous. Under this assumption, any bounded connected component in the union of particles contains a distributionally unique $h$-maximal thinning. This suggests that if the particle process is in a subcritical regime, where with probability 1 the union of all particles does not contain an infinite connected component, then all elements of $\mc{T}_{h,\ms{\ell-max}}$ have the same distribution. Moreover, this distribution can be constructed as a factor from the law of the reference particle process by a simple thinning rule: Consider the connected components of the particle processes separately and inside each of them choose the almost surely uniquely determined $h$-maximal thinning.

More precisely, writing $\{\Phi_i\}_{i\ge1}$ for the collection of connected components of the union of particles in $\Phi$, we let $\Phi_{\ms{max}}$ denote the thinning of $\Phi$ obtained by selecting in each of the $\Phi_i$ the (almost surely uniquely determined) $h$-maximal hard-core subset. Hence, $\Phi_{\ms{max}}$ is a stationary hard-core thinning of $\Phi$. We show that every locally maximal thinning has the same distribution as $\Phi_{\ms{max}}$.

\begin{theorem}
\label{uniqGenThm}
Assume that the union of particles in $\Phi$ almost surely does not percolate and that all factorial moment measures of the marked point process $\{(c(K),h(K)):\,K\in\Phi\}$ are absolutely continuous. Then, every locally maximal thinning has the same distribution as $\Phi_{\ms{max}}$.
\end{theorem}

As we have seen in the discussion preceding Theorem~\ref{uniqGenThm}, dropping the absolute continuity assumption might destroy uniqueness of locally $h$-maximal distributions in the sense of Theorem~\ref{uniqGenThm}. In contrast, the effects of moving from the sub- to the supercritical regime are less clear. 

To make this more precise, we consider the Poisson Boolean model. That is, $\Phi$ is assumed to be a homogeneous Poisson particle process with some intensity $\gamma_0 \in (0,\infty)$ and some grain distribution $\Q_0$. 
Let $\gamma_{\ms{c}}=\gamma_{\ms{c},\Q_0}$ denote the critical intensity for continuum percolation in the Poisson Boolean model with $\Q_0$-distributed grains. If $\gamma_0 <\gamma_{\ms{c}}$, then, by Theorem~\ref{uniqGenThm}, there exists a distributionally unique stationary $h$-maximal hard-core thinning.
Next, we note that in many situations, there is at least a small range above $\gamma_{\ms{c}}$, where distributional uniqueness continues to hold. Indeed, elementary geometric reasoning shows that in some specific configurations, it is possible to single out particles that can never be part of a volume-maximal thinning. After disregarding them, the intensity of the remaining relevant particles becomes strictly smaller, so that the essential enhancement technology~\cite{strictGrim,strictIneq} brings one back into a subcritical regime. 
\begin{theorem}
\label{uniqBallThm}
Assume that $\Phi$ is a Poisson Boolean model of balls with intensity $\gamma_0 \in (0, \infty)$, and whose radius distribution is absolutely continuous with support bounded away from $0$ and $\infty$. Moreover, on the support of the radius distribution, the Lebesgue density is assumed to be bounded away from $0$. Then there exists $\gamma_{\ms{u}}>\gamma_{\ms{c}}$ such that if $\gamma_0 <\gamma_{\ms{u}}$, then all locally volume-maximal particle processes have the same distribution.
\end{theorem}
The general approach outlined in~\cite{strictIneq} is sufficiently flexible to be applicable to the problem described in Theorem~\ref{uniqBallThm}, although some care is needed to transfer the geometric constructions in~\cite{strictIneq} to the setting of random radii.

What happens for intensities that are substantially larger than the critical intensity $\gamma_{\ms{c}}$? Even on a heuristic level, it is not entirely clear what kind of behavior is expected. At first glance, the breaking of rotational symmetry in models from statistical physics~\cite{symBreak2,symBreak} could suggest that at high intensities a certain form of crystallization occurs. However, on the contrary, the philosophy of~\cite{gamWeight} predicts that long-range dependencies could disappear through substantial fluctuations in the grain sizes.

In order to put these speculations on a more rigorous foundation, we provide examples of functions $h$ and supercritical Poisson Boolean models of spherical grains with arbitrarily high intensity for which distributional uniqueness of locally $h$-maximal thinnings holds. In these examples, $h$ is a specific functional that is increasing in the grain volume, but exhibits substantially more pronounced fluctuations than the volume. In particular, this example is genuinely different from the essentially subcritical case discussed in Theorem~\ref{uniqGenThm}. 

More precisely, consider a Poisson Boolean model $\Phi$ in $\R^d$ with intensity $\gamma_0 > 0$ and random radii distributed uniformly on the interval $[0,1]$. Now, for $a\ge1$ we put 
$h_a(B_r(x))=\exp(ar)$. Our last main result shows distributional uniqueness of locally maximal thinnings holds for all sufficiently $a$.

\begin{theorem}
\label{highDensProp}
Assume that $\Phi$ is a Poisson Boolean model of balls with intensity $\gamma_0 > 0$, and whose radius distribution is uniform on $[0,1]$. If $a\ge1$ is sufficiently large, 
then all locally $h_a$-maximal particle processes have the same distribution.
\end{theorem}

In the proof of Theorem~\ref{highDensProp}, we will see that for large values of $a$, stationary $h_a$-maximal hard-core thinnings resemble packings based on the random sequential adsorption algorithm. The stabilization techniques used to show that random sequential adsorption is well-defined in a stationary setting~\cite{yuk01,yukichPacking}  play an essential r\^ole to establish distributional uniqueness of locally maximizing thinnings in the setting of Theorem~\ref{highDensProp}.

The rest of the present paper is organized as follows. First, in Sections~\ref{existSec},~\ref{approxProofSec} and~\ref{eqSec}, we prove Theorems~\ref{existenceProp},~\ref{approxThm} and~\ref{localOptThm}, respectively. Next, the issue of distributional uniqueness in the subcritical and barely supercritical regime is considered in Section~\ref{uniqGenSec}. Since the proof of Theorem~\ref{uniqBallThm} is based on delicate elementary geometric results, large parts of it are postponed to an appendix. Finally, in Section~\ref{fluctSec} we prove Theorem~\ref{highDensProp}.

\section{Proof of Theorem~\ref{existenceProp}}
\label{existSec}
In this section, we prove that $\mc{T}_{\ms{i-max}}(\LPhi)\ne\es$. 

\begin{proof}[Proof of Theorem~\ref{existenceProp}]
Since every particle process $\Psi \in \mc{T}_{}$ is a realizationwise subset of $\Phi\circ S$, the family of distributions of particle processes in $\mc{T}$ is tight~\cite[Proposition 11.1.VI]{pp1}. In particular, if $\{\Psi_n\}_{n\ge1}$ is a sequence of elements of $\mc{T}_{\ms{hc}}$ whose $h$-intensities converge to $\gamma_{h,\ms{max}}$, then there exists a subsequence $\{\Psi_{n_i}\}_{i\ge1}$ converging weakly to the distribution of a particle process $\Psi$. We claim that $\Psi\in\mc{T}_{h,\ms{i-max}}$.

Indeed, the family $\sum_{K\in\Psi_{n_i}}\one\{c(K)\in[0,1]^d\}h(K)$ is uniformly integrable since  $\gamma_h(\LPhi)<\infty$. Since the discontinuities of $h$ form a zero-set with respect to $\Q$, the $h$-intensity $\gamma_h(\Psi_{n_i})$ converges to $\gamma_h(\Psi)$ as $i$ tends to $\infty$. It remains to show that $\Psi\in\mc{T}_{\ms{hc}}$. First, note that $E_{\ms{hc}}$ is a closed set, so that by the Portmanteau theorem, 
$$\P(\Psi\in E_{\ms{hc}})\ge \limsup_{i\to\infty}\P_{n_i}(\Psi_{n_i}\in E_{\ms{hc}})=1.$$
	Finally, as shift operations are continuous~\cite[Proposition A2.3.V]{pp1}, weak limits of stationary particle processes are again stationary, which completes the proof.
\end{proof}

\section{Proof of Theorem~\ref{approxThm}}
\label{approxProofSec}
In order to prove Theorem~\ref{approxThm}  we introduce a thinning $\Psi^+_{\Xi(k)}$ of $\Phi$ that is $h$-maximal among all thinnings of $\Phi$ where hard-core constraints are only imposed within the cells. More precisely, let 
$$\Phi^+_{\Xi_i(k)}=\{K\in\Phi:\,c(K)\in\Xi_i(k)\}$$
denote the family of grains in $\Phi$ whose grain center is contained in the cell $\Xi_i(k)$ of $\Xi(k)$. Then, we let $\Phi_{\Xi_i(k),\ms{max}}^+$ denote a hard-core thinning of $\Phi^+_{\Xi_i(k)}$ that maximizes the aggregated $h$-values. In case that the maximizer is not unique, we choose one according to some translation invariant rule. 
		Finally, we put $\Phi_{\Xi(k),\ms{max}}^+=\cup_{i\ge1}\Phi_{\Xi_i(k),\ms{max}}^+$. Note that in contrast to $\Phi^-_{\Xi(k),\ms{max}}$, the thinning $\Phi_{\Xi(k),\ms{max}}^+$ is not necessarily hard core, as we allow overlaps between grains associated with different cells.

Now, we claim that 
\begin{align}
	\label{phiPlusEq}
	\gamma_h(\Psi)\le\gamma_h(\Phi_{\Xi(k),\ms{max}}^+),
\end{align}
holds for any stationary hard-core thinning $\Psi$ of $\Phi$. In particular, $\gamma_{h,\ms{max}}\le \gamma_h(\Phi_{\Xi(k),\ms{max}}^+)$.

	In order to see~\eqref{phiPlusEq}, we extend the original probability space so as to support a stationary pair $(\Psi, \Xi(k))$ such that the distributions of $(\Xi(k), \Phi)$ and $(\Psi, \Phi)$ do not change. We may choose $\Xi(k)$ and $\Psi$ conditionally independent given $\Phi$.
	Now,~\eqref{phiPlusEq} can be shown using either the refined Campbell formula or the mass-transport principle~\cite{gentner11,gp}. Following the latter approach, every cell $\Xi_j \in \Xi(k)$ transports mass $h(K_i)$ to every grain $K_i \in \Psi$ whose center $c(K_i)$ is contained in $\Xi_j$. Since the $\Xi_j$ form a partition, the total mass received by the grain $K_i$ is $h(K_i)$. On the other hand, the total mass sent out by the cell $\Xi_j$ equals
	$$\sum_{K_i\in\Psi}\one\{c(K_i) \in \Xi_j\}h(K_i).$$
%
	Therefore, by the mass-transport principle~\cite[Corollary 3.9]{gp},
	\begin{align*}
		\gamma_h(\Psi)&=\E\sum_{\Xi_j\in \Xi(k)}\one\{c(\Xi_j)\in[0,1]^d\}\sum_{K_i\in\Psi}\one\{c(K_i)\in\Xi_j\}h(K_i).
    \end{align*}
    In particular, the definition of $\Phi_{\Xi(k),\ms{max}}^+$ yields that
    \begin{align*}
				      \gamma_h(\Psi)&\le\E\sum_{\Xi_j\in \Xi(k)}\one\{c(\Xi_j)\in[0,1]^d\}\sum_{K_i\in\Phi_{\Xi(k),\ms{max}}^+}\one\{c(K_i)\in\Xi_j\}h(K_i)\\
				      &=\gamma_h(\Phi_{\Xi(k),\ms{max}}^+),
	\end{align*}
	which proves~\eqref{phiPlusEq}.

	In the light of~\eqref{phiPlusEq}, it suffices to show that 
	$$\lim_{k\to\infty}\gamma_h(\Phi_{\Xi(k),\ms{max}}^+)-\gamma_h(\Phi^-_{\Xi(k),\ms{max}})=0.$$
In order to establish this bound, we first show that inside the cell $\Xi_i(k)$ the difference of $h$-values between $\Phi_{\Xi(k),\ms{max}}^+$ and $\Phi^-_{\Xi(k),\ms{max}}$ is of the order of the surface area of the cell $\Xi_i(k)$.

\begin{lemma}
\label{surfLem}
Almost surely, for every $i,k\ge1$, it holds that 
$$\sum_{\substack{K\in\Phi_{\Xi_i(k),\ms{max}}^+}}h(K)-\sum_{\substack{K\in\Phi_{\Xi_i(k),\ms{max}}^-}}h(K)\le\sum_{\substack{K\in\Phi:\,c(K)\in \Xi_i(k)\\ K\cap\partial\Xi_i(k)\ne\es}}h(K).$$
\end{lemma}
\begin{proof}
	The definition of $\Phi^-_{\Xi_i(k)}$ provided in Section~\ref{defSec} gives that
	$$\sum_{K\in\Phi_{\Xi_i(k),\ms{max}}^+\cap\Phi_{\Xi_i(k)}^-}h(K)\le\sum_{\substack{K\in\Phi^-_{\Xi_i(k),\ms{max}}}}h(K).$$
Hence,
\begin{align*}
	\sum_{\substack{K\in\Phi_{\Xi_i(k),\ms{max}}^+}}h(K)&\le\sum_{K\in\Phi_{\Xi_i(k),\ms{max}}^+\cap\Phi_{\Xi_i(k)}^-}h(K)+\sum_{\substack{K\in\Phi:\,c(K)\in \Xi_i(k)\\ K\cap\partial \Xi_i(k)\ne\es}}h(K)\\
														 &\le\sum_{\substack{K\in\Phi^-_{\Xi_i(k),\ms{max}}}} h(K)+\sum_{\substack{K\in\Phi:\,c(K)\in \Xi_i(k)\\ K\cap\partial \Xi_i(k)\ne\es}}h(K),
\end{align*}
as required.
\end{proof}
Using Lemma~\ref{surfLem}, we can now prove Theorem~\ref{approxThm}.
\begin{proof}[Proof of Theorem~\ref{approxThm}]
First, by	Lemma~\ref{surfLem},
\begin{align*}
	&\gamma_h(\Phi^+_{\Xi(k),\ms{max}})-\gamma_h(\Phi^-_{\Xi(k,\ms{max}})\\
	&\quad\le\E\sum_{i\ge1}\sum_{\substack{K\in\Phi\\ c(K)\in[0,1]^d\cap\Xi_i(k)}}\one\{K\cap\partial\Xi_i(k)\ne\es\}h(K)\\
	&\quad = \E\sum_{\substack{K\in\Phi\\ c(K)\in[0,1]^d}}\one\{K\cap \cup_{i\ge1}\partial\Xi_i(k)\ne\es\}h(K),
\end{align*}
	where $\Xi_i(k)$ denotes the $i$th cell of the tessellation $\Xi(k)$. Now, we choose $m\ge1$ large and distinguish the cases whether or not $K$ is contained in $[-m,m]^d$. Thus,
\begin{align*}
	&\gamma_h(\Phi^+_{\Xi(k),\ms{max}})-\gamma_h(\Phi^-_{\Xi(k),\ms{max}})\\
	&\quad\le\E\sum_{\substack{K\in\Phi\\ c(K)\in[0,1]^d}}\one\{K\not\subset[-m,m]^d\}h(K)\\
	&\quad\phantom{\le}+\E\Big[\one\{[-m,m]^d \cap \cup_{i\ge1} \partial \Xi_i(k)\ne\es\}\sum_{\substack{K\in\Phi\\ c(K)\in[0,1]^d}}h(K)\Big].
\end{align*}
	The dominated convergence theorem implies that the first expected value tends to 0 as $m\to\infty$. By uniform integrability, it remains to show that the indicator in the second expectation tends to $0$ in probability as $k\to\infty$. Moreover, by sub-additivity and the Campbell formula~\cite[Theorem 3.1.2]{sWeil} for the process of cells $\Xi(k)$,
\begin{align*}
	\P([-m,m]^d\cap\cup_{i\ge1}\partial\Xi_i(k)\ne\es)&=\E\lambda_d([0,1]^d\cap\cup_{i\ge1}\partial\Xi_i(k)\oplus [-m,m]^d)\\
	    &\le\E\sum_{i\ge1}\lambda_d([0,1]^d\cap\partial\Xi_i(k)\oplus[-m,m]^d)\\
	    &=\E\sum_{\substack{i\ge1\\c(\Xi_i(k))\in[0,1]^d}}\lambda_d(\partial\Xi_i(k)\oplus[-m,m]^d)\\
	    &=\frac{\E\lambda_d(\partial\Xi(k)^*\oplus[-m,m]^d)}{\E\lambda_d(\Xi(k)^*)}.
\end{align*}
Now, sending $k\to\infty$ and using the assumption in Theorem~\ref{approxThm} completes the proof.
\end{proof}

As a first example, we may apply Theorem~\ref{approxThm} to the family $\{\Xi_{\ms{Pois}}{(k)}\}_{k\ge1}$, where $\Xi_{\ms{Pois}}{(k)}$ denotes a Poisson-Voronoi tessellation based on a homogeneous Poisson point process that is independent of $\Phi$ and has intensity $k^{-1}$, where $k\ge1$. Clearly, the pair $(\Xi_{\ms{Pois}}{(k)},\Phi)$ is jointly stationary and we now provide a simple scaling argument to show that if $k$ is sufficiently small, then the typical isoperimetric-type coefficient becomes arbitrarily small.
\begin{lemma}
	\label{isopIndLem}
	Let $k\ge1$ and $\Xi_{\ms{Pois}}{(k)}$ be a Voronoi tessellation based on a homogeneous Poisson point process with intensity $k\ge1$. Then, for every $m\ge1$ the typical isoperimetric-type coefficients in~\eqref{isoEq} tend to 0 as $k \to \infty$.
\end{lemma}
\begin{proof}
	The Mecke-Slivnyak Theorem~\cite[Theorem 9.4]{lecturesPoisson} implies that the typical cell $\Xi_{\ms{Pois}}(k)^*$ is obtained as the zero-cell after adding an additional point at the origin. Additionally, by the scaling property of the Poisson point process,
	$$\frac{\E\lambda_d(\partial\Xi_{\ms{Pois}}(k)^*\oplus[-m,m]^d)}{\E\lambda_d(\Xi_{\ms{Pois}}(k)^*)}=\frac{\E\lambda_d(\partial\Xi_{\ms{Pois}}(1)^*\oplus[-m/k^{1/d},m/k^{1/d}]^d)}{\E\lambda_d(\Xi_{\ms{Pois}}(1)^*)}.$$
    Note that $\E\lambda_d(\partial\Xi_{\ms{Pois}}(k)^*\oplus[-m,m]^d)$ is finite by~\cite[Theorem 2]{hug}.
	In particular, the assertion follows from the dominated convergence theorem.
\end{proof}
The Voronoi tessellation used in Lemma~\ref{isopIndLem} makes use of additional randomness in the sense that the point process of cell centers is assumed to be independent of the particle process $\Phi$. However, by resorting to a factor construction provided by \'A.~Tim\'ar~\cite{timar}, it is possible to dispense with this additional randomness and produce a sequence of Voronoi tessellations as factors of $\Phi$. More precisely, assuming the group of isometries of the process of particle centers $\{c(K)\}_{K\in\Phi}$ to be almost surely trivial, in~\cite{timar} a sequence of Voronoi tessellations $\{\Xi_{\ms{fact}}(k)\}_{k\ge1}$ is constructed such that for every $m\ge1$ there exists $a=a(m)>0$ with the property that for every $k\ge1$ and every cell $\Xi_{\ms{fact},i}(k)$ of $\Xi_{\ms{fact}}(k)$,
\begin{enumerate}
	\item the pair $\big(\Xi_{\ms{fact}}(k),\Phi\big)$ is jointly stationary, 
	\item $\Xi_{\ms{fact}}(k)$ can be expressed as measurable function of $\Phi$, and
	\item $\frac{\lambda_d(\partial\Xi_{\ms{fact},i}(k)\oplus[-m,m]^d)}{\lambda_d(\Xi_{\ms{fact},i}(k))}\le a2^{-k}$.
\end{enumerate}

Using item 3. it is straightforward to verify  that the typical isoperimetric-type coefficients tend to 0.
\begin{lemma}
\label{isopFactLem}
For the family of tessellations $\{\Xi_{\ms{fact}}(k)\}_{k\ge1}$ the typical isoperimetric-type coefficients tend to 0 for every $m \ge 1$.
\end{lemma}
\begin{proof}
	Combining the definition of the typical cell with item 3. gives that
\begin{align*}
	\frac{\E\lambda_d(\partial\Xi_{\ms{fact}}(k)^*\oplus[-m,m]^d)}{\E\lambda_d(\Xi_{\ms{fact}}(k)^*)}&=\frac{\E\sum_{c(\Xi_{\ms{fact},i}(k))\in [0,1]^d}\lambda_d(\partial\Xi_{\ms{fact},i}(k)\oplus[-m,m]^d)}{\E\sum_{c(\Xi_{\ms{fact},i}(k))\in [0,1]^d}\lambda_d(\Xi_{\ms{fact},i}(k))}\\
	&\le a2^{-k},
\end{align*}
which tends to $0$ as $k\to\infty$.
\end{proof}

\section{Proof of Theorem~\ref{localOptThm}}
\label{eqSec}
The proof of Theorem~\ref{localOptThm} consists of two directions that are presented in Sections~\ref{meanLocSec} and~\ref{locMeanSec}, respectively. 
\subsection{Local maximality from intensity-maximality}
\label{meanLocSec}
In this subsection we fix a stationary hard-core thinning $\Psi$ of $\Phi$.
If local maximality fails, then  we can swap a finite number of grains from the thinning with a finite number of grains outside the thinning in such a way that the hard-core property is preserved and such that there is a net increase in $h$-values. We show that such a swap can be implemented in a translation-covariant manner. This allows us to construct from $\Psi$ a stationary hard-core thinning $\Psi'$ that exhibits a higher $h$-intensity. 

First, it is convenient to introduce short-hand notation for swaps. 
A \emph{swap} is  a pair $\sigma = (\vec{K}, \vec{L})$ consisting of finite subsets $\vec{K}$ of $\Psi$ and $\vec{L}$ of $\Phi \setminus \Psi$. Moreover, the swap $\sigma$ is \emph{valid} if i) $\sum_{K\in\vec{K}}h(K)<\sum_{L\in\vec{L}}h(L)$ and ii) $\vec{L} \cup (\Psi \setminus \vec{K}) \in E_{\ms{hc}}$. Furthermore, if $\sigma = (\vec{K},\vec{L})$ and $\sigma'=(\vec{K'},\vec{L'})$ are valid swaps, then $\sigma$ and $\sigma'$ are \emph{compatible} if $L \cap L' = \es$ for all $L\in\vec{L}$ and $L'\in\vec{L'}$. Note that if $\sigma=(\vec{K},\vec{L})$ and $\sigma'=(\vec{K'},\vec{L'})$ are compatible valid swaps, then also the union $\sigma \cup \sigma' := (\vec{K} \cup \vec{K'}, \vec{L}\cup\vec{L'})$ is a valid swap. 

First, starting from a stationary thinning that is not locally $h$-maximal, we give a translation-covariant construction for a family of compatible valid swaps.  For $m\ge1$ and $K \in \Phi$, we select a valid swap $\sigma^{(m)}_K = (\vec{K'}, \vec{L'})$ such that $K',L'\subset c(K) + [-m/2, m/2]^d$ for all $K' \in \vec{K'}$ and $L' \in \vec{L'}$. If such a swap does not exist, then set $\sigma^{(m)}_K$ to be the empty swap, i.e., $\sigma^{(m)}_K = (\es, \es)$. If there are several possibilities, we fix one of them according to an arbitrary translation-covariant rule. 

Now, assume that $\Phi$ is endowed with iid marks $\{u_K\}_{K \in \Phi}$ from the unit interval $[0,1]$. Then, we introduce a specific Mat\'ern-type thinning $\mc{S}^{(m)}(\Phi, \Psi)$ of $\Phi$, where $K \in \Phi$ is contained in $\mc{S}^{(m)}(\Phi, \Psi)$ if and only if $\sigma^{(m)}_K\ne(\es,\es)$ and $u_K>u_{K'}$ for every $K' \in \Phi$ with $\sigma^{(m)}_{K'}\ne(\es,\es)$ and $|c(K)-c(K')|_\infty\le m$. Next, we show that $\mc{S}^{(m)}_{}(\Phi, \Psi)$ has positive intensity provided that the intensity of all $K\in\Phi$ such that $\sigma^{(m)}_K\ne(\es,\es)$ is positive.

\begin{lemma}
\label{descRandLem}
If $m\ge1$ is such that the intensity of all $K\in\Phi$ with $\sigma^{(m)}_K\ne(\es,\es)$ is positive, then also the intensity of $\mc{S}^{(m)}_{}(\Phi, \Psi)$ is positive.
\end{lemma}
\begin{proof}
By local finiteness, there exists an integer $n\ge1$ such that the intensity $\gamma_n$ of all $K\in\Phi$ with $\sigma^{(m)}_K\ne(\es,\es)$ and $\#\{K'\in\Phi:\,|c(K')-c(K)|_\infty\le m\}\le n$ is strictly positive. In particular, the intensity of $\mc{S}^{(m)}(\Phi, \Psi)$ is at least $n^{-1}\gamma_n$, as required.
\end{proof}

Now we establish the first half of Theorem~\ref{localOptThm}. 
\begin{proof}[Proof of Theorem~\ref{localOptThm}, first half]
Assume that $\Psi$ is a stationary hard-core thinning of $\Phi$ such that with positive probability, $\Psi$ is not locally maximal. In particular, $\sigma^{(m)}_K\ne(\es,\es)$ for some $m\ge1$ and $K\in\Psi$. Hence, by Lemma~\ref{descRandLem}, we see that $\mc{S}^{(m)}_{}(\Phi, \Psi)\ne\es$. We show that implementing all swaps in $\mc{S}^{(m)}_{\ms{sw}}(\Phi, \Psi):=\{\sigma^{(m)}_K:\,K\in\mc{S}^{(m)}(\Phi, \Psi)\}$ transforms $\Psi$ into a stationary hard-core thinning $\Psi'$ with $\gamma_h(\Psi')>\gamma_h(\Psi)$. In particular, $\Psi$ is not intensity-maximal. 
    
    To be more precise, let 
$$\Psi_{\ms{rem}}=\{K\in\Psi:\,K\in\vec{K}\text{ for some }(\vec{K},\vec{L})\in \mc{S}^{(m)}_{\ms{sw}}(\Phi, \Psi)\}$$
and 
$$\Psi_{\ms{add}}=\{L\in\Phi\setminus\Psi:\,L\in\vec{L}\text{ for some }(\vec{K},\vec{L})\in \mc{S}^{(m)}_{\ms{sw}}(\Phi, \Psi)\}$$
denote the stationary particle processes of grains that are removed, respectively added by some swap in $\mc{S}^{(m)}_{\ms{sw}}(\Phi, \Psi)$. Then, by compatibility, $\Psi'=(\Psi\setminus\Psi_{\ms{rem}})\cup\Psi_{\ms{add}}$ is a stationary hard-core thinning of $\Phi$. Moreover, since the latter union is disjoint, we deduce that 
\begin{align}
\label{psiPEq}
\gamma_h(\Psi')=\gamma_h(\Psi)-\gamma_h(\Psi_{\ms{rem}})+\gamma_h(\Psi_{\ms{add}}).
\end{align}
%

For each swap $\sigma\in\mc{S}^{(m)}_{\ms{sw}}(\Phi, \Psi)$ define the center
$$c(\sigma)=\frac{\sum_{K\in\Phi:\, \sigma^{(m)}_K=\sigma}c(K)}{\#\{K\in\Phi: \sigma^{(m)}_K=\sigma\}}.$$
Then, $\Psi_{\ms{add}}$ and $\Psi_{\ms{rem}}$ can be regarded as cluster point processes with primary point process $\{c(\sigma)\}_{\sigma\in\mc{S}^{(m)}_{\ms{sw}}(\Phi, \Psi)}$. Note that the clusters in both $\Psi_{\ms{add}}$ and $\Psi_{\ms{rem}}$ are pairwise disjoint.
In particular, the $h$-intensities of $\Psi_{\ms{add}}$ and $\Psi_{\ms{rem}}$ satisfy
$$\gamma_h(\Psi_{\ms{add}})=\E\sum_{(\vec{K},\vec{L}) \in \mc{S}^{(m)}_{\ms{sw}}(\Phi, \Psi)}\one\{c(\vec{K},\vec{L})\in[-1/2,1/2]^d\}\sum_{L\in\vec{L}}h(L),$$
and
$$\gamma_h(\Psi_{\ms{rem}})=\E\sum_{(\vec{K},\vec{L})\in\mc{S}^{(m)}_{\ms{sw}}(\Phi, \Psi)}\one\{c(\vec{K},\vec{L})\in[-1/2,1/2]^d\}\sum_{K\in\vec{K}}h(K),$$
respectively. Since $\sum_{L\in\vec{L}}h(L)>\sum_{K\in\vec{K}}h(K)$ holds for any valid swap $(\vec{K},\vec{L})$,  identity~\eqref{psiPEq} gives that $\gamma_h(\Psi')>\gamma_h(\Psi)$. 
\end{proof}

\subsection{Intensity-maximality from local maximality}
\label{locMeanSec}
In this section, we assume that $\Psi$ is a stationary hard-core thinning of $\Phi$ and that $\Psi$ is almost surely locally maximal, so that $\gamma_h(\Psi)\le\gamma_{h,\ms{max}}$. 

Using the techniques established in Section~\ref{approxProofSec}, we prove that under condition~\eqref{momCond1}, the thinning $\Psi$ is intensity-maximal. More precisely, for $k\ge1$ let $\Xi{(k)}$ denote a Poisson-Voronoi tessellation with intensity $k^{-1}$ that is defined on a common probability space with $\Psi$ and is independent of $\Psi$. If we knew that $\gamma_h(\Phi^-_{\Xi(k),\ms{max}})\le\gamma_h(\Psi)$, then we could let $k$ tend to $\infty$ and apply Theorem~\ref{approxThm} to deduce that 
$$\lim_{k \to \infty} \gamma_h(\Phi^-_{\Xi(k),\ms{max}})=\gamma_h(\Psi)=\gamma_{h,\ms{max}}.$$
Therefore, the proof of Theorem~\ref{localOptThm} is complete once the following auxiliary result is established. 
\begin{lemma}
\label{surfLem2}
Let $k\ge1$ be arbitrary. Then, $\gamma_h(\Phi^-_{\Xi(k),\ms{max}}) \le \gamma_h(\Psi)$.
\end{lemma}
\begin{proof}
	To simplify notation, we write $\Xi$ instead of $\Xi(k)$.  For $m\ge1$ define
$$\Phi_{\Xi,m}=\{K\in \Phi^-_{\Xi,\ms{max}}:\, K\subset [-m/2,m/2]^d\}\cup\{K\in \Psi:\, K\cap [-m/2,m/2]^d=\es\}$$
as the union of all grains of $\Phi_{\Xi,\ms{max}}^{-}$ that are contained in $[-m/2,m/2]^d$ and all grains of $\Psi$ that do not hit $[-m/2,m/2]^d$. Then, $\Phi_{\Xi,m}$ is a hard-core thinning of $\Phi$ that does not contain grains intersecting the boundary $\partial [-m/2,m/2]^d$ of $[-m/2,m/2]^d$. In particular, by local maximality, 
\begin{align*}
\E\sum_{K\in \Phi_{\Xi,m}:\, c(K)\in[-m/2,m/2]^d}h(K)&\le\E\sum_{K\in\Psi:\,K\cap[-m/2,m/2]^d\ne\es}h(K)\\
&\le\E\sum_{K\in\Psi:\,c(K)\in[-m/2,m/2]^d}h(K)\\
&\phantom{\le}+\E\sum_{K\in\Phi:\,K\cap\partial[-m/2,m/2]^d\ne\es}h(K)\\
&=m^d\gamma_h(\Psi)+\E\hspace{-0.2cm}\sum_{K\in\Phi:\,K\cap\partial[-m/2,m/2]^d\ne\es}\hspace{-0.2cm}h(K).
\end{align*}
Moreover, 
\begin{align*}
&m^d\gamma_h(\Phi^-_{\Xi,\ms{max}})-\E\sum_{K\in\Phi_{\Xi,m}:\,c(K)\in[-m/2,m/2]^d}h(K)\\
&\quad=\E\sum_{K\in\Phi^-_{\Xi,\ms{max}}:\,c(K)\in[-m/2,m/2]^d}h(K)-\E\sum_{K\in\Phi_{\Xi,m}:\,c(K)\in[-m/2,m/2]^d}h(K)\\
&\quad\le \E\sum_{K\in\Phi:\,K\cap\partial[-m/2,m/2]^d\ne\es}h(K).
\end{align*}
But by the same arguments as in Section~\ref{approxProofSec},
$$\E\hspace{-0.2cm}\sum_{K\in\Phi:\,K\cap\partial[-m/2,m/2]^d\ne\es}\hspace{-0.2cm}h(K)\le \gamma_0 \E\int\lambda_d(\partial([-m/2,m/2]^d)\oplus (-K))h(K)\Q(\d K).$$
Since the right-hand side is of order $O(m^{d-1})$, sending $m$ to infinity concludes the proof.
\end{proof}

\section{Proof of Theorems~\ref{uniqGenThm} and~\ref{uniqBallThm}}
\label{uniqGenSec}
In this section, we prove distributional uniqueness of locally maximal thinnings in subcritical and barely supercritical regimes of continuum percolation. First, recall that we let $\Phi_{\ms{max}}$ denotes the thinning of $\Phi$ that is obtained by selecting the $h$-maximal hard-core subset in each of the connected components. Since we assume absolute continuity with respect to Lebesgue measure and absence of percolation, this selection is well-defined and unique.
\begin{proof}[Proof of Theorem~\ref{uniqGenThm}]
First, we show that $\Phi_{\ms{max}}$ is locally maximal. Indeed, suppose that $\Psi'$ is a hard-core thinning that is obtained by swapping a finite number of particles in $\Phi_{\ms{max}}$ with a finite number of particles in $\Phi\setminus\Phi_{\ms{max}}$. If $i\ge1$ is such that $\Phi_{\ms{max}}\cap\Phi_i\ne\Psi'\cap\Phi_i$, then the maximality property of $\Phi_{\ms{max}}$ implies that 
\begin{align}
\label{largerEq}
\sum_{K\in\Phi_{\ms{max}}\cap\Phi_i}h(K)>\sum_{L\in\Psi'\cap\Phi_i}h(L).
\end{align}
Hence, $\Phi_{\ms{max}}$ is locally maximal. On the other hand, suppose that $\Psi'$ is a locally maximal hard-core thinning. 
    If $\Psi'\ne\Phi_{\ms{max}}$, then $\Psi'\cap\Phi_i\ne\Phi_{\ms{max}}\cap\Phi_i$ for some $i\ge1$. In particular, when swapping $\Psi'\cap \Phi_i$ with $\Phi_{\ms{max}}\cap\Phi_i$, then inequality~\eqref{largerEq} yields a contradiction to the local maximality of $\Psi'$.
\end{proof}

In the proof of Theorem~\ref{uniqGenThm} we have used that in the subcritical regime of continuum percolation, the hard-core condition is essentially a local condition that does not extend beyond the considered connected component. A priori, this is not necessarily true in the supercritical regime. Nevertheless, slightly above the critical intensity, a refinement of this argument still works, since a positive proportion of the grains cannot be contained in any locally maximal thinning. These grains are called dispensable in the following.

\begin{definition}
\label{dispDef}
A grain $K\in\Phi$ is \emph{dispensable}  if there exists $K'\in\Phi$ such that 
\begin{enumerate}
\item $K\cap K'\ne\es$,
\item $h(K)<h(K')$,
\item if $K''\in\Phi$ is such that $K''\cap K=\es$, then 	$K''\cap K'=\es$.
\end{enumerate}
\end{definition}

Figure~\ref{dispFig} illustrates the definition of dispensable grains.
The crucial observation is that if $\Psi$ is a locally maximal hard-core thinning of $\Phi$, then $\Psi$ does not contain dispensable grains. Indeed, if we remove a dispensable grain $K$ from a hard-core subset and replace it by $K'$, then condition (iii) ensures that the resulting subset is still hard-core, whereas (i) and (ii) show that swapping $K$ and $K'$ increases the $h$-value. This is made precise in Proposition~\ref{uniqBall2Thm} below.

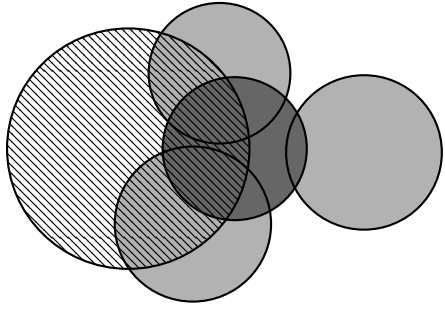
\begin{figure}[!htpb]
\centering
\begin{tikzpicture}[scale=1.0]
\fill[black!30!white] (2.1,-0.05)  circle (1.024);
\fill[black!30!white] (-0.15,-1) circle (1.028);
\fill[black!30!white] (0.2,1) circle (0.932);
\fill[black!60!white] (0.4,0)  circle (0.948);
\fill[pattern=north west lines] (-1,0) circle (1.594);

\draw[thick] (0.4,0)  circle (0.948);
\draw[thick] (2.1,-0.05)  circle (1.024);
\draw[thick] (-0.15,-1) circle (1.028);
\draw[thick] (0.2,1) circle (0.932);
\draw[thick] (-1,0) circle (1.594);

\end{tikzpicture}
\caption{Configuration including a dispensable ball (dark gray)}
\label{dispFig}
\end{figure}

Let $\Phi'$ denote the particle process obtained from $\Phi$ after removing all dispensable grains. Clearly, if $\Phi$ is a Poisson process, then $\Phi'$ has a strictly smaller intensity. The proof of Theorem~\ref{uniqBallThm} proceeds in two steps. First, we show that $\Phi'$ and $\Phi$ have the same locally maximal thinnings. This is a purely deterministic result. In a second step, we show that if we are only slightly above the critical threshold for continuum percolation, then the union of the particles in $\Phi'$ consists of finite connected components almost surely. The proof of the latter claim uses the adaptation of the essential enhancement technique to the continuum setting that has been developed in~\cite{strictIneq}. Finally, we can apply Theorem~\ref{uniqGenThm} to deduce distributional uniqueness of locally maximal thinnings. 

Carrying out this program, we first show that $\Phi$ and $\Phi'$ have the same locally maximal thinnings.
\begin{proposition}
\label{uniqBall2Thm}
Every locally maximal thinning of $\Phi$ is a locally maximal thinning of $\Phi'$, and, vice versa, every locally maximal thinning of $\Phi'$ is a locally maximal thinning of $\Phi$.
\end{proposition}
\begin{proof}
First, let $\Psi$ be a locally maximal thinning of $\Phi$. Then, it suffices to show that $\Psi\subset\Phi'$. If $\Psi\not\subset\Phi'$, then let $K\in\Psi$ be some $h$-dispensable grain. Furthermore, let $K'\in\Phi$ be as in Definition~\ref{dispDef}. In particular, $K\cap K'\ne\es$ implies that $K'\not\in\Psi$. Hence, Definition~\ref{dispDef} shows that swapping $K$ and $K'$ yields another hard-core thinning, despite the fact that $h(K')>h(K)$. This contradicts  local maximality of $\Psi$.

Conversely, let $\Psi$ be a locally maximal thinning of $\Phi'$. Let $\Psi'$ be a hard-core thinning of $\Phi$ differing from $\Psi$ in only finite many grains and satisfying
$$\sum_{K\in\Psi'\setminus \Psi}h(K)>\sum_{K\in\Psi\setminus \Psi'}h(K).$$
Moreover, among all such choices of $\Psi'$, we fix one where $\#(\Psi'\setminus\Phi')$ is minimal. Note that $\Psi'\subset\Phi'$ would result in a contradiction to the local maximality of $\Psi$. Otherwise, choose any $h$-dispensable $K\in\Psi'$. Furthermore, let $K'\in\Phi$ be as in Definition~\ref{dispDef} with the additional requirement that $h(K')$ is as large as possible. In particular, $K'\in\Phi'$. Hence, if we let $\Psi''$ denote the the hard-core thinning of $\Phi$ obtained after swapping $K$ and $K'$, then $\#(\Psi''\setminus\Phi')=\#(\Psi'\setminus\Phi')-1$, contradicting the choice of $\Psi'$.
\end{proof}

Under the assumptions of Theorem~\ref{uniqBallThm}, dispensable grains occur with positive probability. Therefore, the intensity of grains that are relevant for forming locally maximal hard-core subsets is strictly lower than the intensity of the underlying Poisson particle process. Hence, the critical intensity for observing long-range interactions should be strictly higher than the critical intensity of continuum percolation. To provide a rigorous proof of this heuristic, we proceed as in~\cite{strictIneq}, where the essential-enhancement technology that has originally been developed for Bernoulli percolation~\cite{strictGrim,strictStac,strictMensh} has been adapted to various continuum percolation models.

Consider a Poisson particle process $\Phi$ with intensity $\gamma_0 > \gamma_{\ms{c}}$ and grain distribution $\Q$ as in the statement of Theorem~\ref{uniqBallThm}. After possible rescaling, we may assume that the support of the radius distribution is given by the interval $[1, m]$ for some $m>1$, i.e., 
\begin{align}
    \label{mEq}
    \Q(\{K \in \mc{K}:\, B_1(o) \subset K \subset B_m(o)\}) = 1
\end{align}
and $\Q(\{K \in \mc{K}:\, B_{1 + \varepsilon}(o) \subset K \subset B_{m - \varepsilon}(o)\}) < 1$ for every $\varepsilon > 0$. In the following, we assume that $m\ge1.1$ in order to reduce the notational complexity for some geometric auxiliary constructions. The proof of the general case is similar.

First, we declare each grain of the Poisson particle process to be \emph{red} independently with probability $1-p$ for some parameter $p\in(0,1)$. If a grain is not red, then it is \emph{active}. The reason for working with $1-p$ instead of $p$ is that similar as in~\cite[Theorem 2.3]{strictIneq}, we use essential diminishments in the sense that specific grains are suppressed. That is, we want the events related to the occurrence of active paths to be increasing in $p$. Next, it will be convenient to consider dispensable grains with the additional property that they do not intersect too many other grains of the Poisson Boolean model.
To be more precise, an active grain $K$ of the Poisson Boolean model is called \emph{special dispensable} if its radius is smaller than $1.05$ and there exists an active grain $K'$ of the Poisson Boolean model such that the following properties are satisfied.
\begin{enumerate}
\item $K \cap K' \ne \es$,
\item $K$ does not intersect any grains from the Poisson Boolean model of radius smaller than $1.05$,
\item if $K''\in\Phi$ is such that $K''\cap K=\es$, then 	$K''\cap K'=\es$.
\item $K$ intersects at most 3 grains from the Poisson Boolean model of radius larger than $1.05$.
\end{enumerate}
In particular, any special dispensable grain is dispensable in the sense of Definition~\ref{dispDef}. Now, any special dispensable grain is declared \emph{green} independently with probability $1-q$ for some parameter $q\in(0,1)$. 

Let $B_n(o)=\{x\in\R^d:\,|x| \le n\}$ denote the Euclidean ball of radius $n\ge1$ centered at the origin and let $\theta_n(p,q)$ denote the probability, that there exist uncolored grains $K,K'\in\Phi$ such that 
\begin{enumerate}
\item $c(K)\in B_{1}(o)$, $c(K')\in B_n(o) \setminus B_{n-1}(o)$,
\item $K,K'$ can be connected by a chain of overlapping uncolored grains in $\Phi$ with centers in $B_n(o)$.
\end{enumerate}
We put 
$$\theta(p,q)=\liminf_{n\to\infty}\theta_n(p,q)$$
noting that the sequence $(\theta_n(p,q))_{n\ge1}$ is not necessarily decreasing, since to decide which grains are special dispensable we only use grains with center in $B_n(o)$.
The parameters $p$ and $q$ are designed so as to render $\theta(p,q)$ increasing in both $p$ and $q$.
The relation between percolation of the diminished model and $\theta(p,q)$ is given by the following result, whose proof is parallel to~\cite[Proposition 3.1]{strictIneq}. 

\begin{proposition}
\label{thetPercProp}
If $\theta(p,q)=0$, then, almost surely, there is no infinite uncolored connected component.
\end{proposition}
 Our goal is to use the essential-diminishment method to prove the following result.

\begin{proposition}
\label{enhanProp}
There exists $\gamma_{\ms{u}} \in (\gamma_{\ms{c}}, \gamma_0)$ such that $\theta(\gamma_{\ms{u}}/\gamma_0, 0) = 0$. 
\end{proposition}

Before we come to the proof of Proposition~\ref{enhanProp}, let us discuss how it implies Theorem~\ref{uniqBallThm}.
\begin{proof}[Proof of Theorem~\ref{uniqBallThm}]
By Proposition~\ref{uniqBall2Thm}, it suffices to show that there exists some $\gamma_0 > \gamma_{\ms{c}}$ such that, almost surely, $\Phi$ does not contain an infinite connected component after removing the dispensable grains. But this is a consequence of Propositions~\ref{thetPercProp} and~\ref{enhanProp}, which show that almost surely $\Phi$ does not contain infinite connected components already after removing all special dispensable grains.
\end{proof}

The key step in the proof of the essential-diminishment approach is the derivation of the following differential inequality, see~\cite[Lemmas 3.1 and 3.2]{strictIneq}.

\begin{lemma}
\label{diffIneqProp}
There exists a continuous function $\delta: (0,1)^2 \to (0,1)$ such that for all $p,q\in(0,1)$ and $n\ge1$,
\begin{align}
\label{keyEnhLem}
\frac{\partial\theta_n(p,q)}{\partial q} \ge \delta(p,q)\frac{\partial\theta_n(p,q)}{\partial p}.
\end{align}
\end{lemma}

We briefly recall the well-known argument that is used to deduce Proposition~\ref{enhanProp} from Lemma~\ref{diffIneqProp}.
\begin{proof}[Proof of Proposition~\ref{enhanProp}]
Assume that $\gamma_0 > \gamma_{\ms{c}}$ and put $p_{\ms{c}} = \gamma_{\ms{c}} / \gamma_0$. Then, we need to show that 
$$\liminf_{n\to\infty}\theta_n(p_{\ms{c}} + \varepsilon,0) = 0,$$
	for some $\varepsilon > 0$. By Lemma~\ref{diffIneqProp}, there exists $\varepsilon>0$ such that 
	$$\theta_n(p_{\ms{c}} + \varepsilon,1/4) \le \theta_n(p_{\ms{c}} - \varepsilon,1/2)$$
	for all $n\ge1$. This can be seen after a small computation involving derivatives; for a similar problem a detailed derivation is given in the proof of~\cite[Theorem 3.7]{Grim99}. Since an independently thinned Poisson particle process is again a Poisson particle process, the definition of $p_{\ms{c}}$ gives that
$$\liminf_{n\to\infty} \theta_n(p_{\ms{c}}+\varepsilon, 0) \le \liminf_{n\to\infty} \theta_n(p_{\ms{c}} + \varepsilon, 1/4) \le \liminf_{n\to\infty} \theta_n(p_{\ms{c}}-\varepsilon,1/2)=0,$$
as required.
\end{proof}

The proof of Lemma~\ref{diffIneqProp} is delicate, but most of the arguments used in the proof of~\cite[Lemma 3.2]{strictIneq} carry over to the present setting with some extra work. In fact, the only substantial difference is that in our setting, the radii of balls are random, whereas that radius is held constant in~\cite{strictIneq}. Hence, to avoid unduly repetition, we omit the details of the proof of Lemma~\ref{diffIneqProp} and only reproduce the most important steps in the appendix.

\section{Uniqueness through random fluctuations in grain sizes}
\label{fluctSec}

The proof of Theorem~\ref{highDensProp} proceeds in several steps. As in the theory of Gibbs measures, the question of distributional uniqueness is related to the stabilization concept via the notion of disagreement percolation~\cite{disagHc}. That is, loosely speaking, the symmetric difference of two distinct locally maximal thinnings percolates.
This is made precise in the following result, where for a locally finite set of grains $\varphi$, we let $\mc{G}(\varphi)$ denote the contact graph on $\varphi$. That is, $\mc{G}(\varphi)$ has vertex set $\varphi$ and $K, L \in \varphi$ are connected by an edge if $K\cap L \ne \es$. The specific form of $h_a$ is not of importance for the following result.
\begin{lemma}
\label{disagreePerc}
Let $\Psi,\Psi'$ be locally $h_a$-maximal thinnings of $\Phi$.  Then, all connected components of the graph $\mc{G}(\Psi\Delta \Psi')$ are infinite.
\end{lemma}
\begin{proof}
    Assume that $C$ was a finite connected component $\mc{G}(\Psi\Delta\Psi')$ and let $K_0 \in C$ be arbitrary. We show that local $h_a$-maximality implies that $K_0\in\Psi\cap\Psi'$, which is a contradiction to the choice of $K_0$. We assume that 
$$\sum_{K\in C\cap\Psi}h_a(K)>\sum_{K'\in C \cap \Psi'}h_a(K'),$$
noting that the proof proceeds analogously if the inequality is reversed. Defining 
$$\Psi''=(\Psi'\setminus C)\cup(C\cap\Psi),$$
we claim that $\Psi''$ provides a counter-example to the local $h_a$-maximality of $\Psi'$. Indeed, $\Psi' \Delta \Psi''$ is contained in $C$ and therefore finite. Moreover, 
\begin{align*}
\sum_{K''\in\Psi''\setminus \Psi'}h_a(K'')=\sum_{K\in C\cap\Psi}h_a(K)>\sum_{K'\in C\cap\Psi'}h_a(K')=\sum_{K' \in \Psi' \setminus \Psi''}h_a(K'),
\end{align*}
which yields the desired contradiction to the local maximality of $\Psi'$.
\end{proof}

Next, we identify macroscopic point configurations that can act as a shield against disagreement percolation, i.e., percolation  of the graph $\mc{G}(\Psi \Delta \Psi')$ considered in Lemma~\ref{disagreePerc}.  For this purpose, for any integer $a \ge 1$, we say that a grain $K_0$ is \emph{$a$-huge} in a particle configuration $\varphi$ if 
$$h_a(K_0)>\sum_{\substack{K \in \varphi:\, K\cap K_0\ne\es\\ h_a(K)<h_a(K_0)}}h_a(K).$$
That is $h_a(K_0)$ is larger than then the sum of the $h_a$-values of all smaller grains intersecting $K_0$. First, we show that $a^{2d}$-huge grains occur with high probability.
\begin{lemma}
\label{dispHighProb}
The probability that all grains $K \in \Phi$ with $c(K)\in Q_{3a}(o)=[-\tfrac32a, \tfrac32a]^d$ are $a^{2d}$-huge in $\Phi$ tends to $1$ as $a \to \infty$. 
\end{lemma}
\begin{proof}
    By the multivariate Mecke formula~\cite[Theorem 4.4]{lecturesPoisson}, it suffices to show that the probability that an additional grain $K_0=B_{r_0}(o)$, $r_0\in [0,1]$ added at the origin fails to be $a^{2d}$-huge in $\Phi \cup \{B_{r_0}(o)\}$ is of order $o(a^{-d})$. Indeed, letting $U$ denote a uniform random variable that is independent of the Poisson particle process, 
	\begin{align*}
		&\E \#\{K \in \Phi:\, c(K) \in Q_{3a}(o) \text{ and $K$ is not $a^{2d}$-huge in $\Phi$} \}\\
        &\quad  = \lambda\int_{Q_{3a}(o)} \P(B_U(x)\text{ is not $a^{2d}$-huge in $\Phi \cup \{B_U(x)\}$}) \d x \\
		&\quad  = 3^da^d\lambda \P(B_U(o)\text{ is not $a^{2d}$-huge in $\Phi \cup \{B_U(x)\}$}).
	\end{align*}
Let 
$$E_{1,a}=\{\#\{K\in\Phi:\,K\cap K_0\ne\es\}\le a\}$$ 
denote the event that $K_0$ is intersected by at most $a$ grains from $\Phi$. Moreover, we let
$$E_{2,a}=\{r\not\in (r_0-a^{-\tfrac{3d}2},r_0]\text{ for all $K=B_r(x)\in\Phi$ with $K\cap K_0\ne\es$}\}$$ 
denote the event that for every grain $K=B_r(x)$ intersecting $K_0$	the radius $r$ is outside $(r_0-a^{-\tfrac{3d}2},r_0]$ . 

First, we claim that if $E_{1,a}$ and $E_{2,a}$ occur, then $K_0$ is $a^{2d}$-huge in $\Phi$. Indeed, in this case,
$$\sum_{\substack{B_r(x)\in\Phi:\, B_r(x)\cap K_0\ne\es\\ r<r_0}}h_{a^{2d}}(K)\le a\exp(a^{2d}(r_0-a^{-\tfrac32d})) <  \exp(a^{2d}r_0).$$
	Second, the probability of the complements of the events $E_{1,a}$ and $E_{2,a}$ is of order $o(a^{-d})$. Indeed, the probability of $E_{1,a}^c$ can be bounded from above by the probability that the ball $B_{2}(o)$ contains more than $a$ grain centers. As the grain centers form a Poisson point process, this probability decays to $0$ exponentially fast as $a\to\infty$. To bound the probability of $E_{2,a}^c$, we let $N$ denote the number grains of $\Phi$ whose centers are contained in $B_2(o)$. Then, 	noting that $2r_0\le2$, Campbell's formula gives that
\begin{align*}
    \P(E_{2,a}^c)&\le\E\#\{B_r(x) \in \Phi:\, (x, r) \in B_2(o) \times (r_0-a^{-\tfrac{3d}2},r_0]\}=a^{-\tfrac{3d}2}\E N,
\end{align*}
which is of order $o(a^{-d})$ by the finiteness of the Poisson particle intensity.
\end{proof}

Working only with $a$-huge grains simplifies the computation of locally maximal thinnings substantially. For instance, if $K_0$ is $a$-huge and is larger than any grain of $\Phi$ intersecting $K_0$, then $K_0$ is contained in any locally $h_a$-maximal thinning $\Psi$ of $\Phi$. Indeed, otherwise moving from $\Psi$ to 
$$(\Psi\cup\{K_0\})\setminus\{K\in\Psi:\, K\cap K_0\ne\es\}$$
would lead to a net increase of aggregated $h_a$-values, thereby contradicting local $h_a$-maximality of $\Psi$.
This observation can be generalized to multiple particles under consideration.

To capture dependencies between $a$-huge grains, we introduce a directed intersection graph $\mc{G}' = \mc{G}'(\Phi)$ on the vertex set $\Phi$ where we draw an edge from $K = B_r(x)$ to $K' = B_{r'}(x')$ if and only if $K \cap K' \ne \es$ and $r < r'$. The \emph{cluster} $\mc{C}(K) = \mc{C}(K,\Phi)$ of $K \in \Phi$ is defined as the set of all $K'\in\Phi$ that are reachable from $K$ via a directed path in $\mc{G}'$.
\begin{lemma}
\label{depGraphLem}
Let $a \ge 1$, $z \in \Z^d$ and assume that $K'$ is $a^{2d}$-huge in $\Phi$ for every $K'$ with $c(K') \in Q_{3a}(az)$. Let $\Psi, \Psi'$ be locally $h_{a^{2d}}$-maximal thinnings of $\Phi$ and let $K \in \Phi$ be such that $K'' \subset Q_{3a}(az)$ for all $K'' \in \mc{C}(K)$. Then, $K \not \in \Psi \Delta \Psi'$.
\end{lemma}
\begin{proof}
    Assume that $K \in \Psi$. We need to show that $K \in \Psi'$. We proceed by induction on the cardinality of $\mc{C}(K)$. The case $\#\mc{C}(K) = 1$ has already been treated in the discussion preceding the lemma, so that we may assume $\#\mc{C}(K)>1$. Now, let $\{K_1, \ldots, K_m\}$ denote the grains to which an edge is drawn from $K$ in $\mc{G}'$. In order to derive a contradiction, we assume that $K \not \in \Psi'$.
By the hard-core property of $\Psi$, we have $K_i\not \in \Psi$ for every $1 \le i \le m$. Hence, by induction we conclude that also $K_i \not \in \Psi'$ for every $1\le i \le m$. In particular, $\Psi''=(\Psi'\setminus A) \cup \{K\},$
is hard-core, where 
$$A=\{K' \in \Phi \setminus \{K_1,\ldots,K_m\}:\,K' \cap K\ne\es\}$$
	denotes the family of grains different from $K_1,\ldots,K_m$ intersecting $K$. Since $K$ is $a^{2d}$-huge in $\Phi$, we see that moving from $\Psi'$ to $\Psi''$ leads to a net increase in $h_{a^{2d}}$-values, thereby contradicting the assumption of local $h_{a^{2d}}$-maximality of $\Psi'$.
\end{proof}
Now, a site $z \in \Z^d$ is \emph{$a$-good} if 
\begin{enumerate}
	\item $K$ is $a^{2d}$-huge in $\Phi$ for every $K \in \Phi$ with $c(K) \in Q_{3a}(az)$, and 
\item $\mc{C}(K) \subset Q_{3a}(az)$ for every $K \in \Phi$ with $c(K) \in Q_a(az)$. 
\end{enumerate}
Then, as $a \to \infty$, $a$-good sites occur whp.
\begin{lemma}
\label{goodHighProb}
The probability that any given site is $a$-good tends to $1$ as $a \to \infty$.
\end{lemma}
\begin{proof}
	By stationarity, we may assume that the given site is the origin. First, we note that if there exists $K \in \Phi$ with $c(K) \in Q_a(az)$ and $\mc{C}(K) \not \subset Q_{3a}(az)$, then there exists a sequence of $k\ge k_0(a) = \lfloor a/4 \rfloor$ grains $K_0 = K, K_1, \ldots, K_k \in \Phi$ such that 
\begin{enumerate}
\item $|c(K_i) - c(K_{i+1})| \le 2$ for every $0 \le i \le k-1$,
\item $\lambda_d(K_i) \le \lambda_d(K_{i+1})$ for every $0 \le i \le k-1$. 
\end{enumerate}
	By the multivariate Mecke formula, or the form of the factorial moment measures of a Poisson particle process, the expected number of such sequences is bounded above by $a^d(2^d\kappa_d)^{k_0(a)}/k_0(a)!$, which tends to $0$ as $a\to\infty$. An application of Lemma~\ref{dispHighProb} completes the proof.
\end{proof}

After having shown that $a$-good sites occur with high probability, we now prove that the associated cubes can act as shields against disagreement percolation. To make this precise, for $\Gamma\subset\Z^d$ we put $\Gamma^{a}=\cup_{z\in\Gamma}Q_a(az)$.

\begin{lemma}
\label{shieldLem}
Let $\Gamma\subset\Z^d$ be a finite set of $a$-good sites and $K\in\Phi$ be such that $c(K) \not\in \Gamma^a$. Assume that $\Gamma^a$ separates $c(K)$ from $\infty$ in the sense that there does not exist an unbounded curve starting from $c(K)$ and not intersecting $\Gamma^a$. Then, $K \not\in \Psi \Delta \Psi'$ for any locally $h_{a^{2d}}$-maximal thinnings $\Psi,\Psi'$ of $\Phi$.
\end{lemma}
\begin{proof}
If $K \in \Psi \Delta \Psi'$, then Lemma~\ref{disagreePerc} implies that there exists an infinite self-avoiding path $\pi$ in the contact graph $\mc{G}$ that starts from $K$, and has the property that $K' \in \Psi \Delta \Psi'$  for every $K' \in \pi$. Next, by the choice of $\Gamma$, there exist $z_1 \in \Gamma$ and $K_1 \in \pi$ such that $c(K_1) \in Q_a(az_1)$. But this gives a contradiction to Lemma~\ref{depGraphLem}, which necessitates that $K_1 \not \in \Psi \Delta \Psi'$.
\end{proof}

\begin{proof}[Proof of Theorem~\ref{highDensProp}]
The percolation process of $a$-good sites is $4$-dependent. Moreover, by Lemma~\ref{goodHighProb}, the probability that any given site is $a$-good tends to $1$ as $a \to \infty$. Hence, by the dependent-percolation result~\cite[Theorem 0.0]{domProd}, it is possible to choose $a\ge1$ sufficiently large such that with probability $1$, for every $K\in\Phi$ there exists a finite set $\Gamma$ of $a$-good sites such that every unbounded continuous curve starting from $c(K)$ intersects some site of $\Gamma^a$. Now, we may apply Lemma~\ref{shieldLem} to conclude the proof.
\end{proof}

\appendix
\section{Proof of Lemma~\ref{diffIneqProp}}
\label{appSec}
In this section, we sketch the most important ideas needed to prove Lemma~\ref{diffIneqProp}. For details, the reader is referred to~\cite{strictIneq}, especially~\cite[Lemma 3.2]{strictIneq}.
In order to facilitate the transfer, we try to stay as closely as possible to the notation used in~\cite{strictIneq}.

First, it is useful to interpret the partial derivatives appearing in Lemma~\ref{diffIneqProp} in terms of pivotal events by using a formula of Margulis-Russo type, see~\cite[Theorem 19.1]{lecturesPoisson}. Consider a grain $K_0$ whose center $x=c(K_0)$ is contained in $B_n(o)$. This grain is \emph{$1$-pivotal} if an uncolored path of overlapping grains from $B_1(o)$ to $B_n(o)\setminus B_{n-1}(o)$ exists if $K_0$ is active, but such a path does not exist if $K_0$ is red. Similarly, an active and special dispensable grain $K_0$ is \emph{$2$-pivotal} if an uncolored path from $B_{1}(o)$ to $B_n(o)\setminus B_{n-1}(o)$ exists if and only if $K_0$ is not green. The event that $K_0$ is $i$-pivotal is denoted by $E_{n,i}(K_0)$ and we put $P_{n,i}(K_0,p,q)=\P(E_{n,i}(K_0))$. The relation between Lemma~\ref{diffIneqProp} and pivotal probabilities is expressed by the following formula of Margulis-Russo type, see~\cite[Lemma 3.1]{strictIneq} and~\cite[Theorem 19.1]{lecturesPoisson}. 
\begin{lemma}
	\label{mrLem}
Let $n\ge1$ and $p,q\in(0,1)$. Then,
$$
\frac{\partial\theta_n(p, q)}{\partial p} = \gamma_0 \int_{B_n(o)} \int_{\mc{K}}P_{n,1}(x + K,p, q)\Q(\d K)\d x,$$
and
$$\frac{\partial\theta_n(p, q)}{\partial q} = \gamma_0 \int_{B_n(o)} \int_{\mc{K}}P_{n,2}(x + K,p, q)\Q(\d K)\d x.$$
\end{lemma}

Before proving Lemma~\ref{diffIneqProp}, we establish an auxiliary result allowing us to assume that diminishments are suppressed in a fixed annulus around $x$. More precisely, let $A_{\alpha,\beta}(x) = B_\beta(x)\setminus B_{\alpha}(x)$ denote the annulus of outer radius $\beta > 0$ and inner radius $\alpha > 0$ around $x \in \R^d$. We let $R_{n, \alpha, \beta}(x)$ denote the event that all active and special dispensable grains with center in $A_{\alpha, \beta}(x)$ are uncolored.
The proof of the following result is parallel to~\cite[Lemma 3.3]{strictIneq}, but for the convenience of the reader, we provide some details. We recall that the support of the radius distribution is given by $[1, m]$, where $m$ shall denote the constant introduced in~\eqref{mEq}.
\begin{lemma}
\label{noRedLem}
Let $\alpha > 8m$ and $\beta > \alpha + 8m$. Then, there exists a continuous function $\gamma_1 : (0, 1)^2 \to (0,\infty)$ such that
    $$\P(E_{n,1}(x + K) \cap R_{n,\alpha,\beta}(x))\ge \gamma_1(p, q) P_{n,1}(x + K,p,q)$$
    holds for all $n>\beta+8m$, $p\in(0,1)$, $x\in B_{\alpha - 6m}(o)\cup B_n(o) \setminus B_{\beta + 6m}(o)$ and $\Q$-almost all $K \in \mc{K}$.
\end{lemma}
\begin{proof}
We first generate the Poisson Boolean model in $B_n(o)$ and decide for all grains centered outside $A_{\alpha,\beta}(x)$ whether they are red or active. Second, we decide which special dispensable grains centered outside $A_{\alpha,\beta}(x)$ are green. Third, we decide for all grains with center in $A_{\alpha,\beta}(x)$ that intersect more than three other grains or have radius larger than $1.05$ whether they are red or active, noting that such grains can never be special dispensable. Let $W$ denote the remaining grains with centers in $A_{\alpha,\beta}(x)$. 

On the event $E_{n,1}(x + K)$, there exists a coloring of these remaining grains for which $x + K$ is $1$-pivotal. Using this information, we construct the remaining coloring on $W$. If $x + K\in W$ was colored under the coloring provided by $E_{n,1}(x + K)$, then its new color is red. Otherwise, it is active, but not green. In particular, both the old and the new colorings have the same uncolored grains, so that $x + K$ remains $1$-pivotal under the new coloring. Since the number of grains in $W$ is bounded above by some constant, we obtain the desired positive lower bound $\gamma_1(p, q)$ for the probability of observing the coloring.
\end{proof}

In order to prove Lemma~\ref{diffIneqProp}, we need four technical but elementary geometric auxiliary results concerning the existence of configurations of balls exhibiting prescribed intersection patterns. First, we show that grains can be perturbed a little bit without destroying a given intersection pattern. 

\begin{lemma}
\label{elGeomLem-1}
There exists $\delta_1 = \delta_1(d,m)\in(0,1)$ such that for every $\delta\in(0,\delta_1)$ the following holds. Let $V$ be a finite family of balls with radii contained in the interval $[1,m]$. 
Assume that $y_1, y_2\in\R^d$ and $\rho_1,\rho_2\in[1,m]$ are such that $K_1 = B_{\rho_1}(y_1)$ and $K_2=B_{\rho_2}(y_2)$ do not intersect any balls from $V$. Furthermore, suppose that $y_1\not\in K_2$ and that $K_1\cap K_2\ne\es$. Finally, choose $y_2'\in [y_1, y_2]$ such that $|y_2-y_2'|=\delta$. If $\rho'\in(\delta^2 + \rho_2 - \delta,2\delta^2 + \rho_2 - \delta)$ and $y'\in B_{\delta^2}(y_2')$, then $B_{\rho_1}(y_1)\cap B_{\rho'}(y')\ne\es$ and $B_{\rho'}(y')$ does not intersect any ball in $V$.
\end{lemma}
The configuration in Lemma~\ref{elGeomLem-1} is illustrated in Figure~\ref{conf-1Fig}, left.
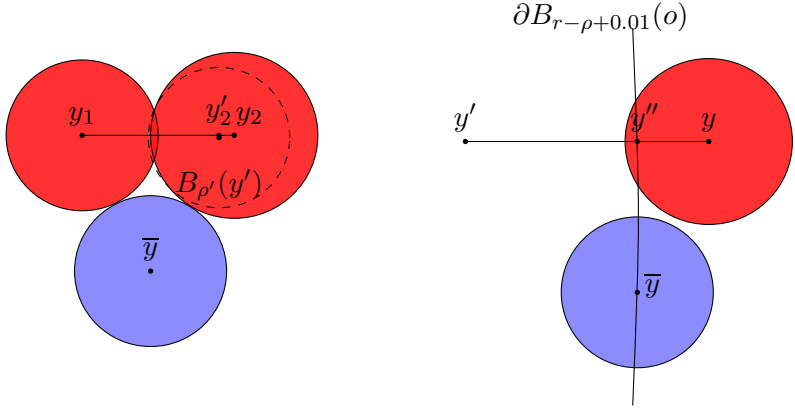
\begin{figure}[!htpb]
\centering
\begin{subfigure}{0.45\textwidth}
\centering
\begin{tikzpicture}[scale=1.0]
\fill[red!90!white,opacity=0.9] (0.0,1) circle (1.1);
\fill[red!90!white,opacity=0.9] (-2.0,1) circle (1.0);
\fill[blue!50!white,opacity=0.9] (-1.1,-0.8) circle (1.0);
\fill (0.0,1) circle (1pt);
\fill (-1.1,-0.8) circle (1pt);
\fill (-2.0,1) circle (1pt);
\fill (-0.2,1) circle (1pt);
\fill (-0.2,0.97) circle (1pt);
\draw (0.0,1)--(-2.0,1);

\draw (0.0,1) circle (1.1);
\draw (-2.0,1) circle (1.0);
\draw (-1.1,-0.8) circle (1.0);
\draw[dashed] (-0.2,0.97) circle (0.93);

\coordinate[label=90:$y_1$] (u) at (-2.0,1);
\coordinate[label=90:$y_2$] (u) at (0.2,1);
\coordinate[label=90:$y_2'$] (u) at (-0.2,1);
\coordinate[label=90:$B_{\rho'}(y')$] (u) at (-0.2,0);
\coordinate[label=90:$\overline{y}$] (u) at (-1.1,-0.8);
\end{tikzpicture}
\end{subfigure}
\begin{subfigure}{0.45\textwidth}
\begin{tikzpicture}[scale=1.0,rotate=-90]
\fill[red!90!white,opacity=0.9] (0.0,1) circle (1.1);
\fill (0.0,1) circle (1pt);
\coordinate[label=90:$y$] (u) at (0.0,1.0);

\draw (0.0,1)--(0.0,-2.2);

\fill (0.0,-2.2) circle (1pt);
\coordinate[label=90:$y'$] (u) at (0.0,-2.2);

\fill (0.0,0.06) circle (1pt);
\coordinate[label=90:$y''$] (u) at (0.0,0.16);

\fill[blue!50!white,opacity=0.9] (2.0,0.06) circle (1.0);
\draw (0.0,1) circle (1.1);
\draw (2.0,0.06) circle (1.0);
\fill (2.0,0.06) circle (1pt);
\coordinate[label=90:$\overline{y}$] (u) at (2.2,0.26);
\draw (-1.5,0) .. controls (1.0,0.1) ..(3.5,0);

\coordinate[label=90:$\partial B_{r-\rho+0.01}(o)$] (u) at (-1.31,-0.42);

\end{tikzpicture}

\end{subfigure}
\caption{Configurations in Lemmas~\ref{elGeomLem-1} and~\ref{elGeomLem0}; blue balls are elements of $V$}
\label{conf-1Fig}
\end{figure}

\begin{proof}
First, we note that $B_{\rho_1}(y_1)\cap B_{\rho'}(y')\ne\es$, since
$$|y_1-y'|\le |y_1-y_2'|+|y_2'-y'|\le\rho_1+\rho_2-\delta+\delta^2 \le \rho_1 + \rho'.$$

    Now, let $B_{\overline{\rho}}(\overline{y})$ be any member of the family $V$. We need to show that $|y_2'-\overline{y}|>\rho_2+\overline{\rho}-\delta+3\delta^2$ provided that $\delta<\delta_1$, where $\delta_1$ is a suitable constant only depending on $d$ and $m$. Put $a=|y_2-\overline{y}|$, $b=|y_1-\overline{y}|$, $c=|y_1-y_2|$, and let 
    $$u = \frac{\langle \overline{y} -y_2, y_1 - y_2\rangle}{ac},$$ 
    where $\langle \cdot, \cdot \rangle$ denotes the standard scalar product in $\R^d$.
    Then,
$$|y_2'-\overline{y}|^2=a^2+\delta^2-2a\delta u,$$
so that it suffices to show 
\begin{align}
\label{aa11eq}
a^2+\delta^2-2a\delta u>(\rho_2+\overline{\rho}-\delta+3\delta^2)^2.
\end{align}
 Note that the left-hand side is increasing in $a$ for $a\ge \delta u$, so that it suffices to prove the claim if $a=\rho_2+\overline{\rho}$. Then,~\eqref{aa11eq} is equivalent to
$$\delta(1-3\delta)>\frac{2a}{a+\rho_2+\overline{\rho}-\delta+3\delta^2}\delta(u - \delta/(2a)).$$
First, the right-hand side of this inequality is at most 
$$\frac{2\delta u}{2-\delta/a}.$$
Additionally, since $\rho_1,\rho_2$ and $\overline{\rho}$ are contained in the interval $[1,m]$, we see that $u$ is bounded above by some $u_1 < 1$ that depends only on $m$ and $d$. To conclude the proof, we choose $\delta_1(d,m) > 0$ such that 
$1-3\delta>\frac{2}{2 - \delta / a} u_1$
whenever $\delta < \delta_1(d,m)$.
\end{proof}

Next, we prove an auxiliary result that will be useful for iteratively creating two families of balls such that no ball from one family intersects a ball from the other set.

\begin{lemma}
\label{elGeomLem0}
There exists $r_1=r_1(d,m)$ with the following property. Let $V$ be a finite family of balls with radii contained in the interval $[1, m]$. Assume that $y\in\R^d$ and $\rho\in[1,m]$ are such that $r = |y| \ge r_1$ and no ball from $V$ has its center in $B_{r - \rho+0.01}(o)$. Furthermore, choose $y' \in [o, y]$ such that $|y' - y| = \rho + 1$ and assume that $B_\rho(y)$ does not intersect any ball in $V$. Then, $y'$ has distance at least $1 + 0.001m^{-1}$ from any ball in $V$.
\end{lemma}

The statement of Lemma~\ref{elGeomLem0} is illustrated in Figure~\ref{conf-1Fig}, right.

\begin{proof}
For simplicity, we assume that $y$ lies on the first coordinate axis, i.e., that $y=re_1$. Next, by simultaneously moving $y$ closer to the origin and decreasing $\rho$, we may reduce the general case to the special one where $\rho = 1$. Now, let $B_{\overline{\rho}}(\overline{y})$ be an arbitrary ball from $V$. Again, by simultaneously moving $\overline{y}$ closer to $y'$ and decreasing $\overline{\rho}$, we may assume that either $\overline{\rho} = 1$ or that $|\overline{y}| = r - 0.99$. 

We start by discussing the first case. If $\langle\overline{y}, e_1\rangle\ge \langle y, e_1\rangle$, then $|\overline{y}-y'|\le 2+0.001m^{-1}$ would enforce that $B_{1}(\overline{y})$ and $B_{1}(y)$ have non-empty intersection. On the other hand, if $\langle\overline{y}, e_1\rangle\le\langle y, e_1\rangle$, then the problem can be reduced to the case where $|\overline{y}| = r - 0.99$, and this case will be considered next.

 By moving $\overline{y}$ on the sphere $\partial B_{r-0.99}(o)$, we may assume that $|y-\overline{y}|=\overline{\rho}+1$. We want to show that $a=|\overline{y}-y'|\ge\overline{\rho}+1+0.001m^{-1}$ holds for a suitable choice of $r_1$. In order to achieve this goal, we first choose $y''\in[o,y]$ such that $|y''|=|\overline{y}|$. Then, we consider the triangle $\Delta y''y'\overline{y}$ and write 
    $$ \alpha = \arccos\frac{\langle \overline{y} ,y''\rangle}{|\overline{y}||y''|},$$ 
so that 
\begin{align*}
a^2 = 1.01^2 + 4|\overline{y}|^2(\sin\tfrac\alpha2)^2 - 4|\overline{y}|(\sin\tfrac\alpha2)1.01(\sin\tfrac\alpha2) = 1.01^2+4|\overline{y}|(\sin\tfrac\alpha2)^2(|\overline{y}| - 1.01).
\end{align*}
 Moreover, by considering the triangle $\Delta yo\overline{y}$, 
\begin{align*}
(\overline{\rho} + 1)^2 &= |\overline{y}|^2 + (|\overline{y}|+0.99)^2 - 2|\overline{y}|(|\overline{y}|+0.99)(1-2(\sin\tfrac{\alpha}{2})^2)\\
&= 0.99^2 + 4|\overline{y}|(\sin\tfrac{\alpha}{2})^2(|\overline{y}| + 0.99).
\end{align*}
Hence, 
$$a^2 = 1.01^2 + \sigma(|\overline{y}|)((\overline{\rho} + 1)^2 - 0.99^2) = 1.01^2 + \sigma(|\overline{y}|)(\overline{\rho}^2 + 2\overline{\rho} + 0.0199).$$
where 
$$\sigma(|\overline{y}|)=\frac{|\overline{y}|-1.01}{|\overline{y}| + 0.99}.$$
In particular, 
\begin{align*}
a^2-(\overline{\rho}+1+0.001m^{-1})^2&=(2.01+0.001m^{-1})(0.01 - 0.001m^{-1})\\
&\phantom{=}+(\sigma(|\overline{y}|)-1)(\overline{\rho}^2+2\overline{\rho}\cdot(1+0.001m^{-1}))\\
&\phantom{=}+\sigma(|\overline{y}|)(0.0199-0.002\overline{\rho}m^{-1}). 
\end{align*}
Since the latter expression is strictly positive, we conclude the proof.
\end{proof}
Proving Lemma~\ref{diffIneqProp} is difficult since we need to produce two `arms', i.e., two mutually non-overlapping families of balls. As in~\cite[Lemma 3.2]{strictIneq}, the construction of these non-overlapping grains is based on delicate elementary geometric arguments. For that purpose, we introduce a geometric auxiliary construction that will be used in Lemmas~\ref{elGeomLem1} and~\ref{elGeomLem2}. Let $x\in\R^d$ and $V$, $T$ be finite and mutually non-overlapping subsets of spherical grains whose radii are contained in the interval $[1,m]$. Now, let $K=B_{\rho_y}(y)$, $L=B_{\rho_z}(z)$ denote the balls of $V$ and $T$ whose centers are closest to $x$ and put $r=|x-y|$. We assume that there are no grains from $T\setminus\{L\}$ that are centered inside $B_r(x)$. In order to simplify the presentation, we assume that $x=o$ and $y=re_1=(r,0,\ldots,0)$.

The general goal of the following auxiliary results is to enlarge both $V$ and $T$ by one grain such that i) the new grains are closer to the origin and ii) grains from $V$ still do not overlap with grains from $T$. First, we consider the situation where $|z|>r$ and there exists some grain $L_0=B_{\rho_{z_0}}(z_0)$ in $T$ such that $|z_0|<r+\rho_{z_0}-0.01$.

\begin{lemma}
\label{elGeomLem1}
There exists $\delta_2>0$ such that for every $\delta\in(0,\delta_2)$ there exists $r_2(\delta)>0$ with the following properties. Let $V$, $T$ and $y$ be as above. 
Assume that  $|z|>r\ge r_2$ and that there exists some grain $L_0=B_{\rho_{z_0}}(z_0)$ with $|z_0|<r+\rho_{z_0}-0.01$. Moreover, choose $y'\in[o,y]$ and $z'\in[o,z_0]$ such that $|y'-y|=\rho_{y}+1$ and $|z'-z_0|=\rho_{z_0}+1$. Finally, put $K'=B_{1+\delta}(y')$. Then $K\cap K'\ne\es$, $K'$ does not intersect any grain from $T$, and there exists $z^*\in B_{3\delta}(z')$ such that $L^*=B_{1+\delta}(z^*)$ has the following properties:
\begin{enumerate}
\item $K'\cap L^*=\es$, 
\item $L_0\cap L^*\ne\es$, 
\item  $L^*$ does not intersect any grain from $V$.
\end{enumerate}
\end{lemma}

The configuration in Lemma~\ref{elGeomLem1} is illustrated in Figure~\ref{conf1Fig}.

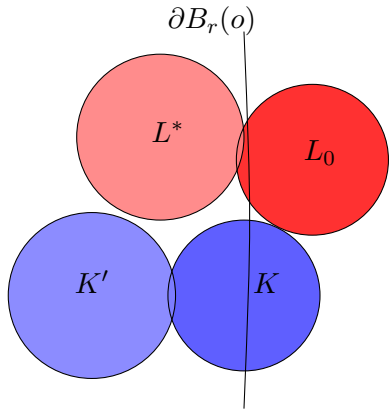
\begin{figure}[!htpb]
\centering
\begin{tikzpicture}[scale=1.0,rotate=-90]
\fill[red!90!white,opacity=0.9] (-1.8,0.9) circle (1.0);
\fill[blue!70!white,opacity=0.9] (0,0) circle (1.0);
\fill[red!50!white,opacity=0.9] (-2.1,-1.1) circle (1.1);
\fill[blue!50!white,opacity=0.9] (0.0,-2.0) circle (1.1);

\draw (-1.8,0.9) circle (1.0);
\draw (0,0) circle (1.0);
\draw (-2.1,-1.1) circle (1.1);
\draw (0.0,-2.0) circle (1.1);

\draw (-3.5,0) .. controls (-1.0,0.1) ..(1.5,0);

\coordinate[label=90:$K$] (u) at (0.1,0.3);
\coordinate[label=90:$K'$] (u) at (0.1,-2.0);
\coordinate[label=90:$L_0$] (u) at (-1.6,1.0);
\coordinate[label=90:$L^*$] (u) at (-1.9,-1.0);
\coordinate[label=90:$\partial B_r(o)$] (u) at (-3.31,-0.42);

\end{tikzpicture}
\caption{Configuration in Lemma~\ref{elGeomLem1}}
\label{conf1Fig}
\end{figure}
\begin{proof}
Without loss of generality, we may assume that $z_0$ lies in the plane spanned by $e_1$ and $e_2$.
 First, applying Lemma~\ref{elGeomLem0} to $y=y$ and $y=z_0$ gives that $y'$ and $z'$ have distance at least $1+0.001m^{-1}$ to any grain from $T$ and $V$, respectively. In particular, $K'$ does not intersect any grain from $T$. Moreover, if $4\delta<0.001m^{-1}$, then any $L^*$ of the form $B_{1+\delta}(z^*)$ with $z^*\in B_{3\delta}(z')$ does not intersect any elements from $V$.

It remains to show that $K'\cap L^*=\es$  and $L_0\cap L^*\ne\es$ hold for a suitable choice of $z^*$. Recalling that $y=re_1$ and writing $v=(z_0-y)/|z_0-y|$ we obtain that 
\begin{align}
\begin{split}
\label{egl1Eq}
|z'-y'|&\ge\langle z'-y',v\rangle=\langle z_0-y,v\rangle+\langle(\rho_y-\rho_{z_0})e_1,v\rangle+\langle z'-z_0+|z'-z_0|e_1,v\rangle \\
&\ge\rho_y+\rho_{z_0}+ \langle(\rho_y-\rho_{z_0})e_1,v\rangle-\delta/2.
\end{split}
\end{align}
provided that $r_2$ is sufficiently large. By moving $z_0$ closer to $o$ and shrinking $B_{\rho_{z_0}}(z_0)$ in such a way that $z'$ remains fixed, we may assume that either $\rho_{z_0} = 1$ or $|z_0|=|y|$.

If $\rho_{z_0} = 1$, then,
    $$\rho_y+\rho_{z_0}+\langle(\rho_y-\rho_{z_0})e_1,v\rangle = (\rho_y - 1)(\langle e_1, v\rangle + 1) + 2\ge2,$$
so that~\eqref{egl1Eq} implies that $|z'-y'|\ge2-\delta/2$. Therefore there exists $z^*\in \partial B_{3\delta}(z')$ such that $|z^*-z_0|=\rho_{z_0}+1$ and $|z^*-y'|> 2+2\delta$. In particular, $L_0\cap L^*\ne\es$ and $K'\cap L^*=\es$.

Finally, assume that $|z_0| = r$. Then~\eqref{egl1Eq} can be strengthened to give $|z'-y'|\ge|z_0-y|-\delta/2$ if $r_2$ is sufficiently large. Now, we conclude as above.
\end{proof}

The second auxiliary result deals with the case, where $|z|\in(r-0.01,r)$ or $\rho_z\le1.01$. Without loss of generality, we assume that $z$ is contained in the two-dimensional plane generated by $e_1$ and $e_2$ and that the $e_2$-coordinate of $z$ is non-negative.

\begin{lemma}
\label{elGeomLem2}
There exist $\delta_3>0$ and $r_3>0$ such that if  $r>r_3$, then the following holds.
Let $V$, $T$, $y$, $z$ be as above, define $y^*=y-(\rho_y+1)2^{-1/2}(e_1+e_2)$ and choose $z^*\in[o,z]$ such that $|z^*-z|=\rho_{z}+1$.
Assume that $L=B_{\rho_z}(z)$ intersects some other grain from $T$ and that any grain $\overline{L}=B_{\rho_{\overline{z}}}(\overline{z})\in T\setminus\{L\}$ satisfies $|\overline{z}|>r+\rho_{\overline{z}}-0.01$. Furthermore, assume that $|z|\in(r-0.01,r)$ or that $\rho_z\in(1,1.01)$. Then, $K^*\cap L^*=\es$, $K^*$ does not intersect any grain from $T$ and $L^*$ does not intersect any grain from $V$, where $K^*=B_{1+\delta_3}(y^*)$ and $L^*=B_{1+\delta_3}(z^*)$. 
\end{lemma}

The statement of Lemma~\ref{elGeomLem2} is illustrated in Figures~\ref{conf2aFig} and~\ref{conf2bFig}.

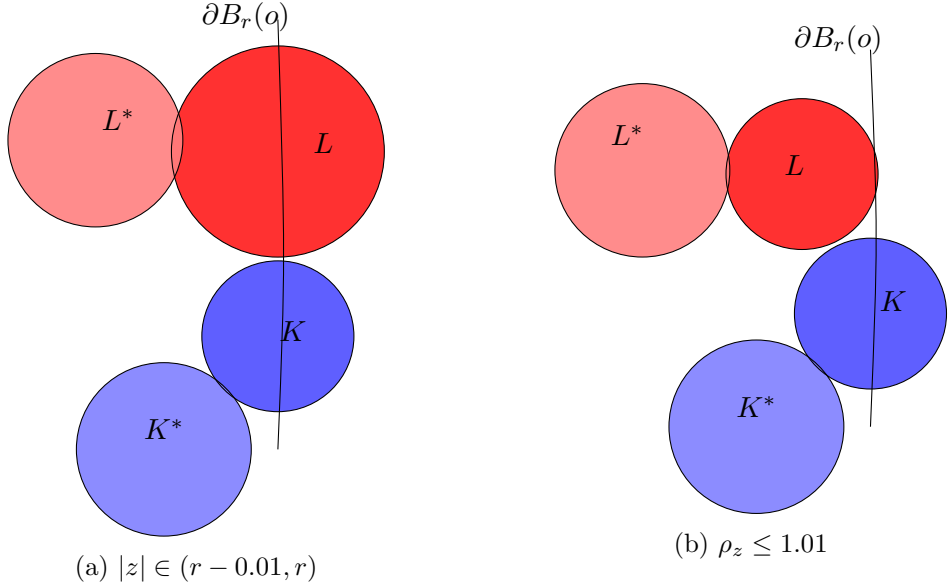
\begin{figure}[!htpb]
\centering
\begin{subfigure}{.45\textwidth}
\centering
\begin{tikzpicture}[scale=1.0,rotate=-90]
\fill[red!90!white,opacity=0.9] (-2.45,0.0) circle (1.4);
\fill[blue!70!white,opacity=0.9] (0,0) circle (1.0);
\fill[red!50!white,opacity=0.9] (-2.6,-2.4) circle (1.15);
\fill[blue!50!white,opacity=0.9] (1.5,-1.5) circle (1.15);

\draw (-2.45,0.0) circle (1.4);
\draw (0,0) circle (1.0);
\draw (-2.6,-2.4) circle (1.15);
\draw (1.5,-1.5) circle (1.15);

\draw (-4.2,0) .. controls (-1.0,0.1) ..(1.5,0);
\coordinate[label=90:$K$] (u) at (0.2,0.2);
\coordinate[label=90:$K^*$] (u) at (1.5,-1.5);
\coordinate[label=90:$L$] (u) at (-2.3,0.6);
\coordinate[label=90:$L^*$] (u) at (-2.6,-2.1);
\coordinate[label=90:$\partial B_r(o)$] (u) at (-3.91,-0.42);
\end{tikzpicture}
\caption{$|z|\in(r-0.01, r)$}
\label{conf2aFig}
\end{subfigure}
\begin{subfigure}{0.45\textwidth}
\centering
\begin{tikzpicture}[scale=1.0,rotate=-90]
\fill[red!90!white,opacity=0.9] (-1.85,-0.9) circle (1.0);
\fill[blue!70!white,opacity=0.9] (0,0) circle (1.0);
\fill[red!50!white,opacity=0.9] (-1.9,-3.0) circle (1.15);
\fill[blue!50!white,opacity=0.9] (1.5,-1.5) circle (1.15);

\draw (-1.85,-0.9) circle (1.0);
\draw (0,0) circle (1.0);
\draw (-1.9,-3.0) circle (1.15);
\draw (1.5,-1.5) circle (1.15);

\draw (-3.5,0) .. controls (-1.0,0.1) ..(1.5,0);
\coordinate[label=90:$K$] (u) at (0.1,0.3);
\coordinate[label=90:$K^*$] (u) at (1.5,-1.5);
\coordinate[label=90:$L$] (u) at (-1.7,-1.0);
\coordinate[label=90:$L^*$] (u) at (-2.1,-3.2);
\coordinate[label=90:$\partial B_r(o)$] (u) at (-3.31,-0.42);
\end{tikzpicture}
\caption{$\rho_z\le1.01$}
\label{conf2bFig}
\end{subfigure}
\caption{Configuration in Lemma~\ref{elGeomLem2}}
\end{figure}

\begin{proof}
We only deal with the case $|z|\in(r-0.01,r)$, since the arguments for the case $\rho_z\in(1,1.01)$ are very similar.
Write $H^+$ and $H^-$ for the subsets of $\R^d$ consisting of points with positive and negative second coordinate, respectively. Then, for sufficiently small $\delta$ and sufficiently large $r_3$ the construction of $y^*$ and $z^*$ implies that  $K^*\subset H^-$ and $L^*\subset H^+$. In particular, $K^*\cap (L\cup L^*)=\es$. Moreover, again if $\delta_3$ is sufficiently small and $r_3$ is sufficiently large, then the distance from $y^*$ to $\partial B_r(o)$ is larger than $1.02$. In particular, $K^*$ does not intersect any grain from $T$. It remains to show that $L^*$ does not intersect any grain from $V$. Let $\overline{K}=B_{\rho_{\overline{y}}}(\overline{y})\in V$ be arbitrary. Then,
    $$|\overline{y} - z^*|^2 = |z^*-z|^2+|\overline{y} - z|^2-2{\langle z^*-z, \overline{y} - z \rangle}.$$
Since $|z^*-z|\ge2$ and $|\overline{y} - z|\ge1+\rho_{\overline{y}}$, it remains to show that
\begin{align}
\label{elGeomEq21}
    (1+\delta+\rho_{\overline{y}})^2\le 4+(1+\rho_{\overline{y}})^2-4(1+\rho_{\overline{y}})\frac{\langle z^*-z, \overline{y} - z \rangle}{|z^* - z||\overline{y} - z| }.
\end{align}
Now, 
    $$4(1+\rho_{\overline{y}})\frac{\langle z^*-z, \overline{y} - z \rangle}{|z^* - z||\overline{y} - z| } \le 0.1$$
provided that $r_3$ is sufficiently large. Hence, a direct computation gives~\eqref{elGeomEq21} for sufficiently small $\delta$.
\end{proof}

Using Lemmas~\ref{elGeomLem-1}--\ref{elGeomLem2}, we now outline the proof of Lemma~\ref{diffIneqProp}. In order to make it easier for the reader to look up further details, we try to adhere closely to the structure of~\cite[Lemma 3.2]{strictIneq}.
\begin{proof}[Proof of Lemma~\ref{diffIneqProp}]
By Lemma~\ref{mrLem} it suffices to construct a continuous function $\delta:(0,1)^2\to(0,1)$ such that 
	$$P_{n,2}(x+K,p,q)\ge \delta(p,q)P_{n,1}(x+K,p,q)$$
holds for all $n\ge1$, $p,q\in(0,1)$, $x \in B_n(o)$ and  $\Q$-almost all $K \in \mc{K}$. 
    We distinguish three cases depending on the distance from $x$ to the origin and to the boundary of $B_n$. In the following, we let $\delta_0\in(0,1)$ and $r_0\ge1$ denote a small and a large constant whose value is determined in the course of the proof. Now, we fix $x \in B_n(o)$ and $K \in \mc{K}$. In the proof we will use the short notation $C_r=B_r(x)$, $E_{n,1} = E_{n,1}(x + K)$, $R_{n, \alpha, \beta} = R_{n,\alpha,\beta}(x)$ and $A_{\alpha,\beta}=A_{\alpha,\beta}(x)$.
   
{\bf Case $\mathbf{|x|\in (r_0,n-r_0)}$.}
 We build up the Poisson Boolean model iteratively. In the first step, we create all grains whose centers are not contained in $C_{0.5r_0}$, and we also determine which of those grains are active. In particular, this information is sufficient to determine the special dispensable grains whose center is not located in $C_{0.6r_0}$. In the second step, we determine which of those grains are green. The partially colored process of grains resulting from these two steps is called $\Phi^{(1)}$.

Now, consider two specific subsets of grains in $T,V\subset \Phi^{(1)}$. The set $T$ consists of those grains in $\Phi^{(1)}$ that are connected by an uncolored path of overlapping grains in $\Phi^{(1)}$ to $B_n(o)\setminus B_{n-1}(o)$. Similarly, $V$ denotes the set of those grains in $\Phi^{(1)}$ that are connected by an uncolored path to $B_1(o)$. In the next step, we continue building up the Poisson Boolean model by adding further grains with centers in $C_{0.5r_0}$. More precisely, the grains are added in decreasing order in the distances of the grain center to $x$. This construction is continued until an active grain $K = B_{\rho_y}(y)$ is found that intersects some grain in $V$ or some grain in $T$. Without loss of generality, assume the former and add $K$ to $V$. We put $r=|x-y|$. 

    Note that if $E_{n,1} \cap R_{n,0.4r_0,0.6r_0}$ occurs, then $K$ exists and no grain of $V$ intersects some grain of $T$. The geometric construction presented in~\cite[Lemma 3.2]{strictIneq} now proceeds by showing that with a probability bounded away from $0$ it is possible to extend the sets $V$ and $T$ radially towards $x$ in a way that $K_0$ becomes $2$-pivotal. We explain how the first steps of this construction can be transferred to the setting of random radii. In fact, the first steps are the most difficult part of the construction since all possible configurations of $V$ and $T$ have to be taken into account accordingly. In contrast, after the first steps one has a fairly precise control on the existing configuration, which makes it much easier to extend $V$ and $T$ radially towards $x$. More precisely, we say that two balls $K^*$ and $L^*$ of radius in $[1+\delta_0,m]$ \emph{extend $V$ and $T$} if 
\begin{enumerate}
\item the centers of $K^*$ and $L^*$ are contained in $C_{r}$,
\item $K^*\cap L^*=\es$,
\item $K^*$ does not intersect any grain from $T$,
\item $L^*$ does not intersect any grain from $V$,
\item $K\cap K^*\ne\es$,
\item $L^*$ intersects some grain from $T$.
\end{enumerate}
In order to guarantee that extensions occur with positive probability, it is important to note that if $\delta_0$ is sufficiently small, then Lemma~\ref{elGeomLem-1} allows small fluctuations of the locations and sizes of $K^*$ and $L^*$ without destroying their configurational properties.

    On the event $E_{n,1} \cap R_{n,0.4r_0,0.6r_0}$, we distinguish three different types of configurations for $V$ and $T$ that will be denoted by $E_3$, $E_4$ and $E_5$, respectively. First, let $E_3$ be the event that there exists some grain $L' = B_{\rho_{z'}}(z')$ in $T$ with $|z'|<r+\rho_{z'}-0.01$. In this case, provided that $\delta_0$ is sufficiently small and $r_0$ is sufficiently large, Lemma~\ref{elGeomLem1} shows that $V$ and $T$ can be extended by suitable grains $K^*$ and $L^*$. Second, let $E_4$ denote the event that $E_3$ does not occur, but there exist $\widetilde{z}\in C_r$ and $\widetilde{\rho}\in[1.01,m]$ such that $\widetilde{K} = B_{\widetilde{\rho}}(\widetilde{z})$ intersects some grain in $T$, but no grain in $V$. By shrinking $\widetilde{K}$ towards the boundary of $C_r$, we may assume that either $|\widetilde{z}| \in (r-0.01,r)$ and $\widetilde{\rho} \ge 1.01$ or that $|\widetilde{\rho}|=1.01$. In both situations, Lemma~\ref{elGeomLem2} allows us to extend $V$ and $T$. Finally, if $E_3$ and $E_4$ do not occur, then we continue the radial generation of the Poisson Boolean model, where we only generate active grains that intersect some grain in $T$ but no grain in $V$. This is done until the first such grain $L = B_{\rho_z}(z)$ is found, which is then added to $T$. We let $E_5$ denote the event that such a grain does exist. Since $E_4$ does not occur, we conclude that $\rho_{z} \le 1.01$. Hence, again provided that $\delta_0$ is sufficiently small and $r_0$ is sufficiently large, Lemma~\ref{elGeomLem2} implies that also under $E_5$ the families $V$ and $T$ can be extended.

	To summarize, we have seen that conditioned on each of the events $E_3$, $E_4$ or $E_5$ the probability of being able to extend each $V$ and $T$ by one further grain is strictly larger than $0$. As in~\cite[Lemma 3.2]{strictIneq}, a refinement of the geometric construction shows that by performing several extension steps, $V$ and $T$ can be extended so that $K_0$ is $2$-pivotal with a probability that is bounded below by a strictly positive function that is continuous in $(p, q)$. Since Lemma~\ref{noRedLem} gives a corresponding lower bound for the probability of $E_{n,1} \cap R_{n,0.4r_0,0.6r_0}$, this concludes the first case.  

    {{\bf Case $\mathbf{|x|\le r_0}$.}} As before, we build up the Poisson Boolean model in different phases. First, we add all grains with centers in $B_n(o)\setminus C_{2r_0}$ and determine which of these are active. In particular, this determines the special dispensable grains in $B_n(o)\setminus C_{3r_0}$ and we find out which of them are green. In the second phase, continue to add grains in decreasing distance to $x$ until an active grain $K$ that is connected to $\partial B_n(o)$ is found. Let $H$ denote the event that such a grain is found which has the additional property that the distance of its center to $x$ is between $1.9r_0$ and $2.1r_0$. Now, the event $H$ occurs if $E_{n,1} \cap R_{n, 1.5r_0, 3r_0}$ occurs. As in the previous case, a suitable geometric construction shows that conditioned on $H$, the event that $K_0$ is $2$-pivotal has a probability that is bounded below by a positive continuous function in $p$. Note that, the present case is substantially simpler since only one  family (and not two families) of grains needs to be extended.

{\bf Case $\mathbf{|x|\ge n-r_0}$.} This case is omitted, as it is similar to the previous one.
\end{proof}

\bibliography{../template.bib}
\bibliographystyle{abbrv}
\end{document}